\documentclass[a4paper,11pt,english]{amsart}
\usepackage{amsmath,amsthm,amsfonts,amssymb,upgreek,multicol,mathrsfs,pifont,amscd,latexsym,cite,bm,geometry,color,fancyhdr,abstract,framed,appendix}
\usepackage{graphicx,algorithm,algorithmic,tabularx,url,float,booktabs,array,threeparttable,tikz}
\usetikzlibrary {arrows.meta}
\usepackage{subcaption}
\usepackage{indentfirst}

\usetikzlibrary{positioning}
\usetikzlibrary{angles,quotes}

\theoremstyle{definition}
\newtheorem{theorem}{Theorem}[section]

\newtheorem{lemma}[theorem]{Lemma}

\newtheorem{ex}[theorem]{Example}

\newtheorem{pro}{Problem}
\numberwithin{equation}{section}
\numberwithin{table}{section}
\numberwithin{figure}{section}

\newcommand{\divt}{\operatorname{div}}
\newcommand{\tdiv}{ \underaccent{\tilde}{\divt}}
\newcommand{\Div}{\operatorname{Div}}
\newcommand{\tDiv}{ \underaccent{\tilde}{\Div}}
\newcommand{\dD}{\operatorname{divDiv}} %
\newcommand{\tdD}{\underaccent{\tilde}{\dD}}
\newcommand{\nab}{\nabla}
\newcommand{\tnab}{\underaccent{\tilde}{\nab}}
\newcommand{\nabs}{{\nabla_{\mathrm{s}}}}  
\newcommand{\tnabs}{{\underaccent{\tilde}{\nabs}}} 
\newcommand{\tpar}{\underaccent{\tilde}{\partial}}

\newcommand{\rd}{\mathrm{d}}
\newcommand{\be}{\boldsymbol{e}}
 
\newcommand{\tbe}{\underaccent{\tilde}{\be}} 
\newcommand{\bft}{\boldsymbol{f}}  %
\newcommand{\tf}{\underaccent{\tilde}{f}} 
\newcommand{\tbf}{\underaccent{\tilde}{\bft}} 
\newcommand{\bl}{\boldsymbol{l}}

\newcommand{\tbl}{\underaccent{\tilde}{\bl}}
\newcommand{\bn}{\boldsymbol{n}}
\newcommand{\tn}{\underaccent{\tilde}{n}}
\newcommand{\tbn}{\underaccent{\tilde}{\bn}}

\newcommand{\ts}{\underaccent{\tilde}{s}}
\newcommand{\bt}{\boldsymbol{t}}
\newcommand{\ttt}{\underaccent{\tilde}{t}} %
\newcommand{\tbt}{\underaccent{\tilde}{\bt}}
\newcommand{\bu}{\boldsymbol{u}}
\newcommand{\tu}{\underaccent{\tilde}{u}}
\newcommand{\tbu}{\underaccent{\tilde}{\bu}}
\newcommand{\bv}{\boldsymbol{v}}
\newcommand{\tv}{\underaccent{\tilde}{v}}
\newcommand{\tbv}{\underaccent{\tilde}{\bv}}

\newcommand{\bCt}{\mathcal{C}}  %
\newcommand{\bCi}{\bCt^{-1}}
\newcommand{\bCO}{\bCt_1}
\newcommand{\bCOi}{\bCO^{-1}}
\newcommand{\bCT}{\bCt_2}
\newcommand{\bCTi}{\bCT^{-1}}
\newcommand{\rE}{\mathrm{E}}

\newcommand{\bK}{\boldsymbol{K}}
\newcommand{\tK}{\underaccent{\tilde}{K}}
\newcommand{\tbK}{\underaccent{\tilde}{\bK}}
\newcommand{\bM}{\boldsymbol{M}}
\newcommand{\tM}{\underaccent{\tilde}{M}}
\newcommand{\tbM}{\underaccent{\tilde}{\bM}}
\newcommand{\bN}{\boldsymbol{N}}
\newcommand{\tN}{\underaccent{\tilde}{N}}
\newcommand{\tbN}{\underaccent{\tilde}{\bN}}
\newcommand{\tS}{{\underaccent{\tilde}{S}}}
\newcommand{\tT}{\underaccent{\tilde}{T}}

\newcommand{\bkap}{\boldsymbol{\kappa}}

\newcommand{\tbkap}{\underaccent{\tilde}{\bkap}}
\newcommand{\bkaph}{\boldsymbol{\kappa}_h}

\newcommand{\tbkaph}{\underaccent{\tilde}{\bkap}_h}

\newcommand{\btau}{\boldsymbol{\tau}}

\newcommand{\tbtau}{\underaccent{\tilde}{\btau}}
\newcommand{\btauh}{\boldsymbol{\tau}_h}

\newcommand{\tbtauh}{\underaccent{\tilde}{\btau}_h}
\newcommand{\Tau}{\mathcal{T}_h}
\newcommand{\tTau}{{\underaccent{\tilde}{\mathcal{T}}_h }}

\newcommand{\VLt}{V}     
\newcommand{\VH}{H}   
\newcommand{\VWt}{W}     
\newcommand{\VsL}{ V_{\mathrm{s}}}  
\newcommand{\Vs}{\Sigma} 
\newcommand{\VsH}{H_{\mathrm{s}}}
\newcommand{\VsC}{C_{\mathrm{s}}}

\makeatletter
\newcommand*{\rom}[1]{\expandafter\@slowromancap\romannumeral #1@}
\makeatother

\newcommand{\RomO}{{I_1}}
\newcommand{\RomT}{{I_2}}
\newcommand{\RomR}{{I_3}}
\newcommand{\RomF}{{I_4}}
\newcommand{\tRomO}{\underaccent{\tilde}{\RomO}}
\newcommand{\tRomT}{\underaccent{\tilde}{\RomT}}
\newcommand{\tRomR}{\underaccent{\tilde}{\RomR}}
\newcommand{\tRomF}{\underaccent{\tilde}{\RomF}}

\usepackage{mathtools}
\usepackage{amsbsy}
\usepackage{amsfonts}
\usepackage{amsthm,amsmath,amssymb}
\usepackage{mathrsfs}
\usepackage{amsthm}
\usepackage{amstext}
\usepackage{bm}
\usepackage{color}
\usepackage{accents}
\usepackage{stmaryrd}
\usepackage{cleveref}
\usepackage{graphics,graphicx}
\usetikzlibrary{decorations.markings}

\begin{document}
	\title{Mixed Variational Formulation of Coupled Plates} 
	
	\author {Jun Hu}
	\address{LMAM and School of Mathematical Sciences, Peking University,
		Beijing 100871, P. R. China.\\ Chongqing Research Institute of Big Data, Peking University, Chongqing 401332, P. R. China. hujun@math.pku.edu.cn}
	
	\author {Zhen Liu}
	\address{LMAM and School of Mathematical Sciences, Peking University,
		Beijing 100871, P. R. China.\\ Chongqing Research Institute of Big Data, Peking University, Chongqing 401332, P. R. China.  zliu37@pku.edu.cn }

	\author{Rui Ma}
	\address{Beijing Institute of Technology,
		Beijing 100081, P. R. China. rui.ma@bit.edu.cn}
	
	\author{Ruishu Wang}
	\address{School of Mathematics,
		Jilin University, Changchun,
		Changchun, Jilin 130012, China.
		wangrs\_math@jlu.edu.cn.}
	
	\thanks{The first author was supported by NSFC
		project 12288101. The third author was supported by NSFC project 12301466.}
	
	

	\maketitle
	
	\begin{abstract}
		This paper proposes a mixed variational formulation for the problem of two coupled plates with a rigid {junction}. 
		The proposed mixed {formulation} introduces {the union of} stresses and moments as {an auxiliary variable}, which {are} commonly of great interest in practical applications. 
		The primary challenge lies in determining a suitable {space involving} both boundary and junction conditions of the auxiliary variable. 
		The {theory} of densely defined operators in Hilbert spaces is employed to define {a nonstandard Sobolev space} without the use of trace operators. 
		The well-posedness is established for the mixed formulation. 
		Besides, continuity conditions for stresses and moments under certain regularity assumptions are presented. 
		Based on these conditions, this paper provides a framework {of} conforming {mixed} finite element methods. {Numerical experiments are given to validate the theoretical results.}  
	\end{abstract}


	\section{Introduction}
	\label{intro}
	
	Elastic multi-structures are multidimensional assemblies {composed of} bodies, plates, and rods {through} suitable junctions.
	These structures are extensively utilized in automotive and aerospace engineering, see, e.g., \cite{timoshenko1959theory, Structuralofshells1982}.
	Based on the principle of minimum potential energy,  Feng and Shi \cite{Fengshi1996} proposed the mathematical theory of multi-structures, involving displacements as fundamental variables. 
	Huang, Shi, and Xu \cite{jianguo2005some} derived the mathematical model in vector forms and proved a generalized Korn’s inequality for elastic multi-structures consisting of bodies, plates, and rods.
	There are some other studies which {{employ}} asymptotic analysis to investigate junctions and equations of multi-structures, see, e.g., \cite{ciarlet1990new, Ciarlet1989JMPA,kozlov1999asymptotic}.
	
	
	This paper considers a specific multi-structure: two coupled plates with a rigid junction.
	The elastic plate model consists of a plane elasticity model in the longitudinal direction 
	and a Kirchhoff plate model in the transverse direction.
	Mathematical models of two coupled plates with rigid and elastic junctions were considered in \cite{bernadou1989numerical,lie1992mathematical,lie1993mathematical, Ciarlet1997, KolpaAndriaMark_ZZAMM_2015}.
{	Most of those results focus on formulating the two coupled plates model in terms of displacements.}
	{In \cite{RafetZulehCMAME2019}, a mixed formulation {which can preserve the continuity of bending moments with the normal direction} on standard $H^1$ spaces for Kirchhoff shells was proposed.}
	By introducing the union of stresses and moments as an auxiliary variable, this paper establishes a {new} mixed variational formulation for two coupled plates with a rigid junction. 
	There are several reasons behind considering {the new mixed formulation.}
	Firstly, stresses and moments are often of higher interest in applications. 
	The mixed formulation allows for the direct calculation of these variables. 
	Secondly, based on the prior {researches} \cite{bernadou1989numerical,lie1992mathematical,lie1993mathematical}, 
	it is known that the junction conditions contain {certain} continuity {requirements} of stresses and moments. 
	The mixed formulation considered in this paper {precisely} preserves this continuity.
	
{{For the plane elasticity model of a single plate, }the Hellinger-Reissner variational principle has been widely used which seeks stresses in $H(\divt,\mathbb{S})$ \cite{ArnoldWinther2002}. Recently, a mixed formulation which seeks moments in $H(\dD,\mathbb{S})$ for the Kirchhoff plate model has been proposed in \cite{Paulydividv, ArnoldFoundations}.
	The mixed variational formulation of two coupled plates is not a straightforward generalization of a single plate.
	The main challenge lies in providing an appropriate {{Sobolev space}} for the union of stresses and moments in the mixed formulation.
	The {junction} conditions for stresses and moments need to be included in the definition of the nonstandard Sobolev space, except for the clamped and free boundary conditions.
	By extending the idea presented in \cite{Walter2018}, this issue is indirectly addressed by employing the theory of densely defined operators within Hilbert spaces. 
	This approach avoids the technical complexities of directly using trace operators in nonstandard Sobolev spaces.
	The well-posedness of the mixed variational problem is proven by utilizing the equivalence to {the variational problem based on displacements} \cite{bernadou1989numerical,lie1992mathematical}. 
	Moreover, continuity conditions {of} stresses and moments are provided under certain regularity assumptions, 
	which are consistent with \cite{bernadou1989numerical,jianguo2005some}. 

	This paper provides a framework of conforming mixed finite element methods based on the continuity conditions of stresses and moments. {Under this framework, various finite elements can be chosen. 
	For finite elements of $H(\divt,\mathbb{S})$, see \cite{HuZhang2014, ArnoldWinther2002} on triangular meshes and \cite{arnold2005rectangular, hurect2015, chen2011conforming } on rectangular meshes. 
	For finite elements of $H(\dD,\mathbb{S})$, see \cite{HuMaZhang2021, YeZhangRAM2022, chen2023new} on triangular meshes and \cite{fuhrer2023mixed}  on rectangular meshes. 	Mixed finite elements in higher dimensions can be found in \cite{Hu2015, HuZhang2015,HuZhang2016,hurectany2014,hurect2015,adams2004mixed, arnold2008finite,awanou2012two} of $H(\divt, \mathbb{S})$ and in \cite{HuMaZhang2021,ChenHuang2022,hu2022new,chen2023new} of $H(\dD, \mathbb{S})$.}  }
To easily implement the continuity conditions,  
this paper uses the $H(\text{div}, \mathbb{S})$ element \cite{HuZhang2014} and the $H(\text{divDiv}, \mathbb{S})$ element \cite{ChenHuang2022} to approximate stresses and moments, respectively.
	Numerical experiments are provided to validate the theoretical results.
	Interested readers can refer to \cite{huang2011NMPDE, HuangGuoShi_CMA_2007,huang2006finite,bernadou1989numerical,lie1992mathematical,lie1993mathematical, LaiHuang_DCDSSB_2009,LaiHuangShi2010,jianguo2005some} for finite element methods of the variational formulation based on displacements.

	The contents of this paper can be outlined as follows. Section 2 recalls the plane {elasticity} model and the {Kirchhoff bending} model. The variational formulations of two coupled plates based on displacements and its well-posedness {are} presented. Section 3 proposes the mixed variational formulation of two coupled plates with the nonstandard Sobolev space. The well-posedness of the mixed formulation is provided. Section 4 gives a framework of {conforming mixed finite element methods}. Numerical experiments are performed to validate the theory.

	
	\section{Preliminaries}\label{Pre}

	This section presents the model assumptions and notation conventions in this paper. By considering the deformation of a single plate, common symbols in linear elasticity are introduced. Following that, the mathematical descriptions for the deformation of two coupled plates are given. In the end, the displacement-based variational {formulation of two coupled plates model} with {a} rigid junction is provided.
	
	%
	
	\subsection{Hypotheses and notations}	
	\label{hypotheses}
	Let $H^{m}(D;X)$ denote the Sobolev space of functions within domain $D$, taking values in space $X$, of $L^{2}(D; X)$ whose distributional derivatives up to the order $m$ also belong to $L^{2}(D; X)$. Similarly, let $C^m(D; X)$ denote the space of $m$-times continuously differentiable functions and $P_{l}(D; X)$ denote the set of all the polynomials with the total degree no greater than $l$, taking values in $X$. In this paper, $X$ could be $\mathbb{R}, \mathbb{R}^{2}$, $\mathbb{M}$ or $\mathbb{S}$, where $\mathbb{M}$ denotes the set of all $\mathbb{R}^{2\times 2}$ matrices while $\mathbb{S}$ denotes the set of all symmetric $\mathbb{R}^{2\times 2}$ matrices. If $X=\mathbb{R}$, then $H^{m}(D)$ abbreviates $H^{m}(D; X)$, similarly for $C^m(D)$. The standard Sobolev norm $\|\bullet\|_{m,D}$ will be taken. When $m=0, H^0(D;X)$ is exactly $L^2(D;X)$. The $L^2$ inner product is denoted by $(\bullet, \bullet)_D$ for the scalar, vector-valued, and tensor-valued $L^2$ spaces over $D$. In particular, on the tensor-valued $L^2$ space,
	$$
	(\boldsymbol{\sigma}, \boldsymbol{\tau})_D \coloneqq \int_D \boldsymbol{\sigma}: \boldsymbol{\tau} \mathrm{d} x=\int_D \sum_{i, j=1}^2 \sigma_{i j} \tau_{i j} \mathrm{d} x, \quad \forall \boldsymbol{\sigma}, \boldsymbol{\tau} \in L^2(D;\mathbb{M}).
	$$
	{Let $\langle \bullet , \bullet \rangle_{U^{*} \times U}$ denote the duality product in $U^{*} \times U$ for a Hilbert space $U$ with its dual $U^{*}$.} For scalar functions $v$, vector-valued functions $\boldsymbol{\psi}$, and matrix-valued functions $\boldsymbol{N}$ with all components in $H^1(\Omega)$, the first-order differential expressions
	$
	\nabla v, \nabla \boldsymbol{\psi}, \nabs \boldsymbol{\psi}, \operatorname{div} \boldsymbol{\psi}
	$
	and $\operatorname{Div} \boldsymbol{N}$
	are defined as follows
	\begin{equation}
		\label{operator}
		\begin{aligned}
			&\nabla v  \coloneqq\left(\begin{array}{l}
				\partial_1 v \\
				\partial_2 v
			\end{array}\right),  ~
			\nabla \boldsymbol{\psi}  \coloneqq\left(\begin{array}{ll}
				\partial_1 \psi_1 & \partial_2 \psi_1 \\
				\partial_1 \psi_2 & \partial_2 \psi_2
			\end{array}\right),  
			~ \nabs \boldsymbol{\psi}  \coloneqq \frac12( \nabla \boldsymbol{\psi} +\nabla \boldsymbol{\psi} ^{T}),
			\\
			&\operatorname{div} \boldsymbol{\psi}  \coloneqq \partial_1 \psi_1+\partial_2 \psi_2, 
			~ \operatorname{Div} \boldsymbol{N}  \coloneqq \left(\begin{array}{l}
				\partial_1 {N}_{11}+\partial_2 {N}_{12} \\
				\partial_1 {N}_{21}+\partial_2 {N}_{22}
			\end{array}\right).
		\end{aligned}
	\end{equation}
	Define the following nonstandard Sobolev spaces
	$$
	\begin{aligned}
		H(\divt, D; \mathbb{R}^2) &\coloneqq \{\boldsymbol{\psi} \in        L^2(D;\mathbb{R}^2) : \; \operatorname{div} \boldsymbol{\psi}  \in L^2(D)\},\\
		H(\Div, D; \mathbb{S}) &\coloneqq \{\bN \in L^2(D; \mathbb{S}) : \; \operatorname{Div} \boldsymbol{N} \in L^2(D;\mathbb{R}^2)\}, \\
		H(\dD, D; \mathbb{S}) &\coloneqq \{\bN \in L^2(D; \mathbb{S}) : \; \dD \boldsymbol{N} \in L^2(D)\}.
	\end{aligned}
	$$
	{The} differential operators in \eqref{operator} {hold} in the weak sense on the corresponding spaces $H^1(D), H^1(D;\mathbb{R}^2),H^1(D;\mathbb{R}^2), H(\divt, D;\mathbb{R}^2), H(\Div, D; \mathbb{S})$.
%
	
	Let $\Omega \subset \mathbb{R}^3$ be a thin plate whose midsurface is denoted as $S$ and lies in the $x$-$y$ plane. For simplicity, assume that $S$ is clamped on $\partial S_0 \subset \partial S$, with nonzero measure, and is free on the complementary part which is denoted as $\partial S_1$, i.e., $\partial S = \partial S_0 \cup \partial S_1,  \partial S_0 \cap \partial S_1 = \varnothing$. The {load} in $S$ is assumed to be a distributed load $(\boldsymbol{f},f_{3})\coloneqq (f_1,f_2,f_3)$. Note that problems with other boundary conditions can be treated similarly. In this paper, the deformations are assumed to be small and governed by the equations of linear elasticity, and the material is supposed to be homogeneous and isotropic. 
	Based on those assumptions, the deformation of $S$ applies to the superposition principle \cite{Ciarlet1997}. 
	This means that the external force can be decomposed into in-plane and out-of-plane components as shown in Figure \ref{decomposition}. 
	
	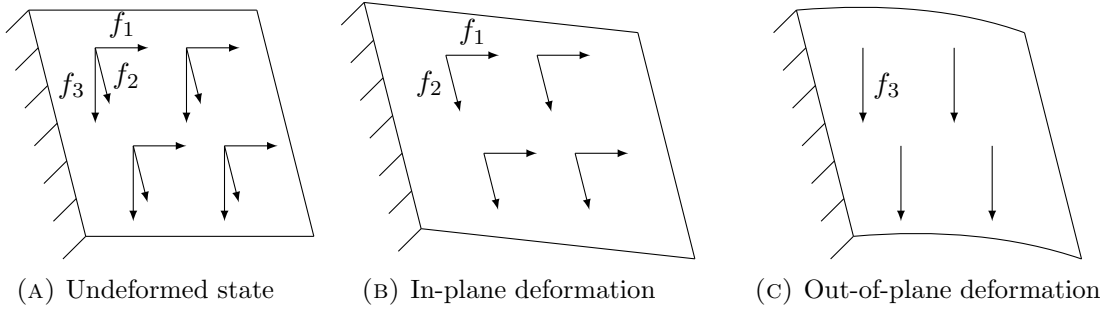
\begin{figure}[ht]
		\centering
		\begin{subfigure}[b]{0.25\textwidth}
			\centering
			\begin{tikzpicture}[>=latex]
				\draw (-0.875,3.5) -- (-0.125,0.5);
				\draw (-0.125,0.5) -- (2.875,0.5);
				\draw (2.875,0.5) -- (2.125,3.5);
				\draw (2.125,3.5) -- (-0.875,3.5);
				\foreach \x/\y in {-0.125/0.5, -0.25/1, -0.375/1.5, -0.5/2,-0.625/2.5, -0.75/3,-0.875/3.5} {
					\draw (\x, \y) -- ++(-0.3,-0.3); 
				}
				\foreach \x/\y in {0/3} {
					\draw[->] (\x, \y) -- ++(0.7,0) node[midway, above]{$f_1$};
					\draw[->] (\x, \y) -- ++(0.1875,-0.75) node[midway,right]{$f_2$};
					\draw[->] (\x, \y) -- ++(0,-1) node[midway, left]{$f_3$}; }
				\foreach \x/\y in {1.2/3,0.5/1.7,1.7/1.7} {
					\draw[->] (\x, \y) -- ++(0.7,0) ;
					\draw[->] (\x, \y) -- ++(0.1875,-0.75) ;
					\draw[->] (\x, \y) -- ++(0,-1); }
			\end{tikzpicture}
			\caption{Undeformed state}
		\end{subfigure}
		\hfill
		\begin{subfigure}[b]{0.3\textwidth}
			\centering
			\begin{tikzpicture}[>=latex]
				\draw (-0.875,3.5) -- (-0.125,0.5);
				\draw (-0.125,0.5) -- (3.475,0.1);
				\draw (3.475,0.1) -- (2.725,3.1);
				\draw (2.725,3.1) -- (-0.875,3.5);
				\foreach \x/\y in {-0.125/0.5, -0.25/1, -0.375/1.5, -0.5/2,-0.625/2.5, -0.75/3,-0.875/3.5} {
					\draw (\x, \y) -- ++(-0.3,-0.3); }
				\foreach \x/\y in {0.2/2.8} {
					\draw[->] (\x, \y) -- ++(0.7,0) node[midway, above]{$f_1$};
					\draw[->] (\x, \y) -- ++(0.1875,-0.75) node[midway,left]{$f_2$};}
				\foreach \x/\y in {1.4/2.8,0.7/1.5,1.9/1.5} {
					\draw[->] (\x, \y) -- ++(0.7,0) ;
					\draw[->] (\x, \y) -- ++(0.1875,-0.75) ; }
			\end{tikzpicture}
			\caption{In-plane deformation}
		\end{subfigure}
		\hfill
		\begin{subfigure}[b]{0.35\textwidth}
			\centering
			\begin{tikzpicture}[>=latex]
				\draw (-0.875,3.5) -- (-0.125,0.5);
				\draw (-0.125,0.5) .. controls (1,0.6) and (2,0.5) ..(2.875,0.2);
				\draw (2.875,0.2) -- (2.125,3.2);
				\draw (-0.875,3.5)  .. controls (0.3,3.6) and (1.3,3.5) .. (2.125,3.2);
				\foreach \x/\y in {-0.125/0.5, -0.25/1, -0.375/1.5, -0.5/2,-0.625/2.5, -0.75/3,-0.875/3.5} {
					\draw (\x, \y) -- ++(-0.3,-0.3); }
				\foreach \x/\y in {0/3} {
					\draw[->] (\x, \y) -- ++(0,-1) node[midway, right]{$f_3$};}
				\foreach \x/\y in {1.2/3,0.5/1.7,1.7/1.7} {
					\draw[->] (\x, \y) -- ++(0,-1); }
			\end{tikzpicture}
			\caption{Out-of-plane deformation}
		\end{subfigure}
		\caption{ (A) illustrates the shape and force distribution under the underformed state;  {(B) and (C)} depict the deformed shapes when subjected solely to in-plane and out-of-plane loads, respectively.}
		\label{decomposition}
	\end{figure}

	
	
	For the in-plane part, the equilibrium equations with boundary conditions can be written as follows
	\begin{equation}
		\begin{aligned}
			-\Div \bN &=\bft, \quad &  & \text{in}~S,\\
			\bu & = 0,  \quad &   &\text{on} ~\partial S_0, \\
			\bN  \bn & = 0, \quad&  & \text{on} ~ \partial S_1,
		\end{aligned}
	\end{equation}
	where $\bn$ is the unit outer {normal} vector of $\partial S$, $\bu = (u_1, u_2)^{T}$ is the displacement vector, and $\bN \in \mathbb{S}$ {is the stress tensor}. It follows from the linear elasticity assumptions that 
	$$
	\begin{aligned}
		\be &= \nabs \bu  = \frac12( \nabla \bu +(\nabla \bu) ^{T}), \\
		\bN &= \bCO \be \coloneqq \frac{Ee}{1-\nu^2}((1-\nu)\be +\nu \text{tr}(\be) \mathbf{I}),
	\end{aligned}
	$$
	{where $\be \in \mathbb{S}$ is the strain tensor, $\mathbf{I}$ is the identity tensor,  $\bCO$ and $\operatorname{tr}(\bullet)$ are the linear compliance tensor and the trace operator, {respectively}. Here $e, E$, and $\nu$ denote the thickness of the plate $\Omega$, the Young's modulus, and the} Poisson {ratio} of the material, respectively.
	
	For the out-of-plane part, taking the Kirchhoff assumptions \cite{bernadou1989numerical}, the deformation is governed by the following equilibrium equation with boundary conditions
	\begin{equation}
		\begin{aligned}
			-\dD \bM &= f_3,\quad &  & \text{in}~S, \\
			u_3= 0,\quad \partial_{n}u_{3}&=0, \quad &&\text{on}~\partial S_0, \\
			T=0, \quad M_{nn}&=0, \quad&& \text{on}~\partial S_1,\\
			\llbracket M_{nt} \rrbracket_x&=0, & & \text{at} ~x \in \mathcal{V}_{ \partial S_1 } {,}
		\end{aligned}
	\end{equation}
	with the deflection $u_{3}$ and {the moment} tensor $\bM \in \mathbb{S}$. The behaviour law  of $S$ is described by
	$$
	\begin{aligned}
		\bK  &= \bK(u_3) \coloneqq -\nabla^2 u_3,\\
		\bM &= \bCT \bK \coloneqq \frac{Ee^3}{12(1-\nu^2)}((1-\nu) \bK+ \nu\text{tr}(\bK )\mathbf{I}),
	\end{aligned}
	$$
	where $\bK \in \mathbb{S}$ is the curvature tensor. Let $\bt$ denote the unit tangential vector of $\partial S$, $\partial_{t}(\bullet) \coloneqq \nabla(\bullet) \cdot \bt$ and $\partial_{n}(\bullet) \coloneqq \nabla(\bullet) \cdot \bn$ be the tangential and normal derivatives, respectively. Define the Kirchhoff shear force (cf. \cite{Fengshi1996}), the normal and twisting moments by
	\begin{equation}
		\label{defT}
		T \coloneqq (\Div \bM)\cdot \bn+\partial_t((\bM \bn)\cdot\bt),\quad M_{nn} \coloneqq (\bM \bn)\cdot\bn,\quad M_{nt}\coloneqq(\bM \bn)\cdot\bt.
	\end{equation}
	{The set $\mathcal{V}_{\partial S_1}$ contains all points $x$ where {exist} two adjacent edges $e_1,e_2 \in \partial S_1$ {satisfying} $e_1 \cap e_2 = x$. These edges do not meet an angle of $\pi$ and possess different tangential and normal vectors $\bt_1,\bn_1$ and $\bt_2,\bn_2$ respectively. The jump term at $x$ is defined by }
	\begin{equation}
		\label{jump}
		\llbracket M_{nt} \rrbracket_x \coloneqq (\bM \bn_1)\cdot\bt_1(x)-(\bM \bn_2)\cdot\bt_2(x), \quad \text{for}~x \in \mathcal{V}_{ \partial S_1 }.
	\end{equation}

	\subsection{Two coupled plates model}
	{This subsection presents} the mathematical description of two coupled plates. Throughout this paper, let $(\underaccent{\tilde}{\bullet})$ denote the quantities related to the plate $\underaccent{\tilde}{\Omega}$ with the same physical meaning as {those} in $\Omega$ without exception. For example, $(\tbf,\tf_3)$ is the distributed load in $\tS$. 	Without loss of generailty, assume that the material parameters $\underaccent{\tilde}{e}, \underaccent{\tilde}{E}, \underaccent{\tilde}{\nu}$ of the plate $\underaccent{\tilde}{\Omega}$ take the same value as $e, E, \nu$ of the plate $\Omega$ for the sake of narrative. 
	
	
	Let $\Omega$ {be} coupled with $\underaccent{\tilde}{\Omega}$ on part of its boundary $\partial \Omega$. Figure \ref{FigJunction} depicts the distinct middle surfaces $S$ and $\tS$ coupled along a common rectilinear junction $\Gamma \coloneqq \partial {S} \cap \partial {\tS}$. Let $\partial S_1$ and $\partial \tS_1$ be the free boundary of the middle surfaces $S$ and $\tS$, respectively. Then it holds that $\partial S = \partial S_0 \cup \partial S_1 \cup \Gamma, \partial \tS = \partial \tS_1 \cup \Gamma$.
	{Let $\bl = \bn \times \bt$}. It follows that $(\bn, \bt, \bl)$ defined on $\Gamma \subset \partial S$ constitutes a local direct orthonormal reference system of $S$. By taking the opposite direction $\underaccent{\sim}{\bt} = -\bt$ on $\Gamma$, the local direct orthonormal reference system in $\tS$ {can} be defined by $(\tbn, \underaccent{\sim}{\bt}, \tbl)$ on $\Gamma \subset \partial \tS$. Define the angle $\theta$ between two plates shown in Figure \ref{FigJunction} by $\cos \theta = \bn \cdot \tbn$. It is easy to verify
	\begin{equation}
		\label{coord}
		\begin{aligned}
			& \tbn =\boldsymbol{n} \cos \theta - \boldsymbol{l} \sin \theta, \quad \underaccent{\sim}{\boldsymbol{t}}=-\boldsymbol{t}, \quad \tbl  = -\boldsymbol{n} \sin \theta - \boldsymbol{l}\cos \theta. \\
		\end{aligned}    
	\end{equation}
	When the plates are coplanar, the angle is equal to $\pi$, and when the plates coincide, the angle is equal to 0. This paper deals with the case $0 < \theta < \pi$.

	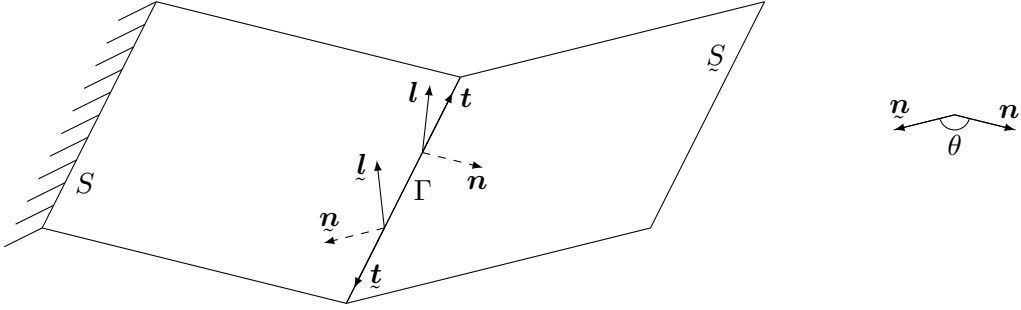
\begin{figure}[ht]
		\centering
		\begin{tikzpicture}[>=latex]
			\draw (0,0) -- (1.5,3)node[pos=0.5, right]{$\Gamma$}; 
			\draw (0,0) -- (4,1);  
			\draw (4,1) -- (5.5,4)  node[pos=0.75, left]{$\tS$};
			\draw (5.5,4) -- (1.5,3);
			\draw (0,0) -- (1.5,3);
			\draw (0,0) -- (-4,1);
			\draw (-4,1) -- (-2.5,4) node[pos=0.2, right]{$S$};
			\draw (-2.5,4) -- (1.5,3);
			\foreach \x in {1,1.3,...,4.3}{
				\pgfmathsetmacro{\resultA}{\x -1}
				\pgfmathsetmacro{\step}{\resultA /0.3}
				\pgfmathsetmacro{\steps}{\step*0.15}
				\pgfmathsetmacro{\resultB}{-4 + \steps}
				\draw (\resultB,\x) -- ++ (-0.5,-0.25);}
			\draw[dashed][->] (1.0, 2) -- ++(0.8,-0.2) node[pos=0.9, below]{$\bn$};
			\draw[->] (1.0,2) -- ++ (0.4,0.8) node[pos=0.9, right]{$\bt$};
			\draw[->] (1.0,2) -- ++ (0.1,0.9) node[pos=0.9, left]{$\bl$};
			
			\draw[dashed][->] (0.5, 1) -- ++(-0.8,-0.2) node[pos=0.9, above]{$\tbn$};
			\draw[->] (0.5,1)-- ++ (-0.4,-0.8) node[pos=0.8, right]{$\tbt$};
			\draw[->] (0.5,1) -- ++ (-0.1,0.9) node[pos=0.9, left]{$\tbl$};
			
			\draw[->] (8, 2.5) -- ++(0.8,-0.2) node[pos=0.9, above]{$\bn$};
			\draw[->] (8, 2.5) -- ++(-0.8,-0.2) node[pos=0.9, above]{$\tbn$};
			\draw (7.2,2.3) coordinate (A)
			-- (8,2.5) coordinate (B)
			-- (8.8,2.3) coordinate (C)
			pic [draw, "$\theta$",angle eccentricity=2, angle radius=0.2cm] {angle};
		\end{tikzpicture}
		\caption{The local coordinate systems of two coupled plates.}
		\label{FigJunction}
	\end{figure}
	


	
	
	Let $\phi \coloneqq (\bu, u_3; \tbu, \tu_3)$ and $F \coloneqq (\bft,f_3, \tbf, \tf_3)$ denote the displacement and external distribute force of the middle surfaces $S\cup \tS$, respectively. Assume that two plates are coupled with a rigid junction. Then the equilibrium equations of two coupled plates,  {see \cite{bernadou1989numerical,lie1992mathematical} for instance}, can be written as follows
	\begin{equation}
		\label{equilibrium}
		\begin{aligned}
			\Div \bN +\bft &=0,&  \dD \bM + f_3&=0 & \text { in } S, \\
			\tDiv \tbN +\tbf&=0, & \tdD \tbM + \tf_3 &=0 & \text { in } \tS,
		\end{aligned}    
	\end{equation}
	with boundary conditions
	\begin{subequations}
		\label{bd}
		\begin{align}
			\bu&=\boldsymbol{0},&   u_{3}&= 0,& u_{3, n}&=0,&  &\text { on } \partial S_0,& \label{bd1}  \\
			\bN  \bn &=\boldsymbol{0},&  T&=0,&  M_{nn} &=0,& & \text { on } \partial S_1,&  \llbracket M_{nt} \rrbracket_x &=0,&  & \forall~x \in \mathcal{V}_{ \partial S_1 },& \label{bd2} \\
			\tbN \tbn &= \boldsymbol{0},&  \tT &=0,&  \tM_{\tn \tn} &=0,&   & \text { on } \partial \tS_1, &\llbracket \tM_{\tn \ttt} \rrbracket_x&= 0,&  & \forall ~x \in \mathcal{V}_{ \partial \tS_1 },& \label{bd3}
		\end{align}    
	\end{subequations}
	the junction conditions
	\begin{subequations}
		\label{jun}
		\begin{align}
			&& u_{3,n}&=-\tu_{3,\tn},&&& &\text {on}~ \Gamma, \label{jun1}\\
			u_{1} &= \tu_{1} \cos \theta - \tu_{3} \sin \theta,&  u_{2} &= -\tu_{2}, & u_{3} &=  - \tu_{1}\sin \theta - \tu_{3}\cos \theta,&&\text {on}~\Gamma, \label{jun2} \\
			&&\tM_{\tn \tn}&= M_{nn},&&& &\text {on}~\Gamma, \label{jun3}\\
			\tN_{\tn\tn} &=  -N_{nn}  \cos \theta + T \sin \theta ,& 
			\tN_{\tn\ttt}&= N_{nt},&  \tT&=  N_{nn}  \sin \theta + T \cos \theta,&&\text {on}~\Gamma ,\label{jun4}
		\end{align}    
	\end{subequations}
	and the corner conditions
	\begin{equation}
		\label{cornercondition}
		\llbracket M_{n t}\rrbracket_x \sin \theta (x)=0, \quad \llbracket\tM_{\tn \ttt} \rrbracket_x - \llbracket M_{n t}\rrbracket_x \cos \theta(x) = 0, \quad \forall ~x \in \mathcal{E}_{\Gamma}.
	\end{equation}
	Here $\mathcal{E}_{\Gamma}$ {denotes the collection of two end points} of $\Gamma$. For $ x \in \mathcal{E}_{\Gamma}$, the jump terms $\llbracket M_{n t}\rrbracket_x$ and $\llbracket\tM_{\tn \ttt} \rrbracket_x$  
	are defined in a similar way as in $\eqref{jump}$ while the two adjacent edges $e_1\cap e_2 = x$ are not in $\Gamma$ simultaneously. {In this case, }the sign in $\eqref{jump}$ depends on the local coordinate systems of the quantity in the bracket. 
	
	
	
	
	\subsection{Variational formulation based on displacements}
	This subsection recalls some results from \cite{bernadou1989numerical, lie1992mathematical}. Introduce the {following} Hilbert spaces:
	\begin{equation}
		\label{VLt}
		\begin{aligned}
		\VLt &\coloneqq \left\{ \psi = \left( \bv, v_3 ; \tbv, \tv_3 \right) : \; \left( \bv, v_3 \right) \right.  \in  L^2(S;\mathbb{R}^2) \times L^2(S), \left( \tbv, \tv_3 \right)  \left.\in L^2\left(\tS;\mathbb{R}^2\right) \times L^2\left(\tS\right)\right\},\\
		\VH &\coloneqq \left\{ \psi = \left( \bv, v_3 ; \tbv, \tv_3 \right) : \; \left( \bv, v_3 \right) \right.  \in H^1(S;\mathbb{R}^2)  \times H^2(S), 
		\left( \tbv, \tv_3 \right) \left.\in H^1\left(\tS;\mathbb{R}^2\right)  \times H^2\left(\tS\right)\right\},
	\end{aligned}
	\end{equation}
	with norms
	\begin{align}
		\|\psi\|_{\VLt} &\coloneqq \left(  \left\|\bv \right\|_{0, S}^2+\left\|v_3\right\|_{0, S}^2+ \left\|\tbv\right\|_{0, \tS}^2+\left\| \tv_3\right\|_{0, \tS}^2 \right)^{\frac12}, \label{normV}\\
		\|\psi\|_{ \VH } &\coloneqq \left( \left\| \bv \right\|_{1, S}^2+\left\|v_3\right\|_{2, S}^2+  \left\| \tbv \right\|_{1, \tS}^2+\left\| \tv_3\right\|_{2, \tS}^2 \right)^{\frac12} {.}  \label{normH}
	\end{align}
	Let $(\bullet, \bullet)_{V}$ denote the $L^2$ inner product in $\VLt$ corresponding to the norm $\| \bullet \|_{V}$. Define 
	\begin{equation}
		\label{W}
		\VWt \coloneqq \left\{\psi=\left( \bv, v_3; \tbv, \tv_3 \right)  \in \VH : \; \psi~\text{satisfies}~\eqref{bd1}, \eqref{jun1}, \text{and}~\eqref{jun2} \right\}.
	\end{equation}
	It is easy to verify that the subspace $\VWt$ is closed in $\VH$. For $\psi \in W$, let $\| \psi\|_{\VWt} = \| \psi \|_{\VH}.$
	
	Given $\phi, \psi \in \VWt$, define
	\begin{equation}
		\label{D}
		\begin{aligned}
			D(\phi, \psi)& \coloneqq \int_S \bN(\bu) : \be(\bv) \mathrm{d} S  +\int_S   \bM(u_3) : \bK(v_3)\mathrm{d} S \\
			& \quad \quad + \int_{\tS}  \tbN(\tbu) : \tbe(\tbv) \mathrm{d} \tS +\int_{\tS} \tbM(\tu_{3}):\tbK(\tv_3) \mathrm{d} \tS.
		\end{aligned}    
	\end{equation}
	Then the variational formulation for two coupled plates with an angle $\theta$ reads as follows.
	\begin{pro} \label{classicalvar}
		{	Given $F \in \VLt,$  find $\phi \in \VWt$ such that 
			$$D(\phi, \psi)=( F,\psi)_{\VLt},\quad  \forall ~\psi \in \VWt.$$}
	\end{pro}
	\begin{lemma}[Well-posedness{\cite[Lemma 2.2]{lie1992mathematical}}]
		\label{coercive}
		{The bilinear form $D(\bullet, \bullet)$ defined in \eqref{D} is continuous and coercive on $\VWt \times \VWt$.} {Moreover, the solution $\phi$ depends continuously on $F$ with a constant $c$, namely, }
		\begin{equation}
			\label{stability}
			\| \phi \|_{\VWt} \leq c \| F\|_{\VLt}.
		\end{equation}
	\end{lemma}
	

	
	\section{Mixed Variational Formulation}
	
	
	This section introduces a mixed formulation of two coupled plates. 
	{The} existence and uniqueness {of the solution} of the mixed variational formulation {is} established, {by showing the} equivalence to the variational {formulation based on displacements}. To derive mixed finite element methods for the mixed {formulation}, this section provides continuity conditions related to the newly introduced space.
	
	\subsection{Spaces and operators}
			This paper introduces
		$\Phi \coloneqq (\bN,\bM; \tbN, \tbM),$ 
		the union of the stresses and moments in \eqref{equilibrium}, as an auxiliary variable for the mixed formulation of two coupled plates.
		{To establish the mixed variational formulation, the main challenge lies in determining a suitable space for $\Phi$ that {incorporates} boundary and junction conditions. Given a stress tensor $\bN \in H(\divt, \mathbb{S})$ on $S$, its trace belongs to $H^{-1/2}(\partial S)$. However, since $H^{-1/2}(\partial S)$ is defined as a dual space, its elements cannot be restricted simply on the junction $\Gamma$ \cite[Section 2.5]{boffi2013mixed}. Similar considerations apply to functions in $H(\dD, \mathbb{S})$ \cite{fuhrer2019ultraweak}. Instead of directly tackling this issue, a new nonstandard Sobolev space is proposed through the application of {adjoint operator} theory. }
		
 	  Define an operator 
 \begin{equation}
 	\label{defB}
 	B \coloneqq (-\nabs, \nab^2; -\tnabs, \tnab^2), 
 \end{equation}
 on $\VWt$ in \eqref{W}. {The} first and the second differential operators are defined in the context of classical weak derivatives. To give {a precise} definition of the operator $(\Div, \dD;$ $\tDiv, \tdD)$ acting on $\Phi$, {employ methodology analogous to that presented in \cite{Walter2018}, which uses the theory of adjoint operators for unbounded operators}. {This approach helps in identifying an appropriate function space that satisfies boundary conditions in a weaker sense.}
	
	Recall {the following classical theory} for densely defined unbounded operators, see, e.g., \cite{boffi2013mixed}. 
	{Let $X$ and $Y$ be Hilbert spaces such that $W$ is dense in $X$, and $X$ and $Y$ will be specified later.} Let the operator $B$ act on $X$ to the dual space of $Y$ as  
	$
	B:   \VWt \subset X \longrightarrow Y^{*}.
	$
	{	Define the operator
		$B^*: D\left(B^*\right) \subset Y \rightarrow X^*,$
		where $D\left(B^*\right)$ is the domain of definition of $B^{*}$, as follows: $y \in D\left(B^*\right)$ if and only if $y \in Y$ and there exists a linear functional $G \in X^*$ such that
		$$
		\langle B x, y\rangle_{Y^{*}\times Y}=\langle G, x\rangle_{X^{*}\times X},\quad \text { for all } x \in \VWt.
		$$
		Define $B^* y=G$. }Note that $\left\langle B^* y, x\right\rangle$ is well-defined for $x \in X$ and $y \in D\left(B^*\right)$ and 
	$$
	\left\langle B^* y, x\right\rangle_{X^{*}\times X}=\langle B x, y\rangle_{Y^{*}\times Y}, \quad \text { for all } x \in \VWt, y \in D\left(B^*\right) .
	$$
	The domain $D\left(B^*\right)$ is a Hilbert space with the graph norm $$\|y\|_{D\left(B^*\right)}\coloneqq \left(\|y\|_Y^2+\left\|B^* y\right\|_{X^*}^2\right)^{\frac{1}{2}}.$$
	
	{Introduce the $L^2$ space}
	\begin{equation}
		\label{defVs}
		\begin{aligned}
		\VsL \coloneqq \left\{ \Psi = \left( \btau, \bkap ; \tbtau, \tbkap \right) : \;   \left( \btau, \bkap \right) \right.  \in L^2(S; \mathbb{S}) \times L^2(S; \mathbb{S}), 
		\left( \tbtau, \tbkap \right)  \left.\in L^2(\tS; \mathbb{S}) \times L^2(\tS; \mathbb{S})\right\},
	\end{aligned}
	\end{equation}
	with the norm 
	\begin{equation}
		\label{normVsym}
		\|\Psi\|_{\VsL} \coloneqq \left(  \left\| \btau \right\|_{0, S}^2+ \left\| \bkap\right\|_{0, S}^2+ \left\|\tbtau \right\|_{0, \tS}^2+\left\| \tbkap \right\|_{0, \tS}^2 \right)^{\frac12}.
	\end{equation}
	{Let $(\bullet, \bullet)_{\VsL}$ denote the $L^2$ inner product in $\VsL$ related to the norm $\| \bullet \|_{\VsL}$.} This paper considers the case $X = \VLt$ and $Y = \VsL$. Obviously, $\VWt \subset \VLt$ and $\VWt$ is dense in $\VLt$. The operator $B$ is then a densely defined {linear operator}. {It follows from \eqref{normV} and \eqref{normVsym} that} $\VLt^{*}=\VLt, \VsL^{*}=\VsL$ and $\langle \bullet, \bullet \rangle_{\VLt^*\times \VLt} = (\bullet,\bullet)_{\VLt}$, $\langle \bullet, \bullet \rangle_{\VsL^*\times \VsL} = (\bullet,\bullet)_{\VsL}$. Define a subspace of $\VsL$ as follows
	\begin{equation}
		\begin{aligned}
			\label{VHs}
			\VsH \coloneqq \{ \Psi = \left( \btau, \bkap ; \tbtau, \tbkap\right): \; & (\btau, \bkap)  \in H(\Div,S;\mathbb{S}) \times H(\dD,S;\mathbb{S}),  \\
			&(\tbtau, \tbkap)\in H(\tDiv,\tS;\mathbb{S}) \times H(\tdD,\tS; \mathbb{S}) \}.
		\end{aligned}
	\end{equation}
	
	According to the previous theory, {$\Phi \in D(B^*)$ if and only if $\Phi \in \VsL$, and there exists a linear functional $G \in \VLt$ such that}
	\begin{equation}
		\label{adj}
		( B \psi, \Phi )_{\VsL} = ( G, \psi )_{\VLt},  \quad \forall~\psi \in W.
	\end{equation}
	Note that $C_{0}^{\infty}(S;\mathbb{R}^2)\times C_{0}^{\infty}(S)\times C_{0}^{\infty}(\tS;\mathbb{R}^2)\times C_{0}^{\infty}(\tS)$ is contained in $W$. This, the definition of \eqref{VHs}, and \eqref{adj} show that $\Phi \in \VsH$ for any $\Phi \in D(B^*)$ and 
	\begin{equation}
		\label{defB*}
		B^{*}= (\Div, \dD; \tDiv, \tdD).
	\end{equation}
	Let $\Sigma \coloneqq D(B^*)$ be the space for the introduced auxiliary variable.
	Then the following equality holds
	\begin{equation}
		\label{adjointequality}
		( B^* \Phi, \psi )_{\VLt} =( B \psi, \Phi )_{\VsL}, \quad   \forall~\psi \in \VWt, \Phi \in \Sigma.
	\end{equation} 
	The nonstandard Sobolev space $\Vs$ can be explicitly given by
	\begin{equation}
		\label{sigma}
		\Vs = \left\{   \Phi \in \VsH : \;  \exists c>0, \text{ such that}\, | (  B \psi, \Phi)_{\VsL} |\leq c\|\psi\|_{\VLt}, \forall \psi\in W  \right\},
	\end{equation}
	equipped with the norm $\| \Phi \|_{\Vs} = \left( \| \Phi\|_{\VsL}^{2} + \| B^{*} \Phi\|_{\VLt}^{2} \right)^{1/2}$, {i.e.}, 
	$$\| \Phi \|_{\Vs} = \left( \| \Phi\|_{\VsL}^{2} + \| \Div \bN \|_{0}^{2} + \| \dD \bM \|_{0}^{2} + \| \tDiv \tbN \|_{0}^2 + \| \tdD \tbM \|_{0}^{2}   \right)^{1/2} .$$ 

	\subsection{Mixed formulation and its well-posedness}
	\label{32}
	
	The results in the previous subsection of {densely} defined operator $B$ and its {adjoint} $B^{*}$ together with their {definition} of domain motivate the new mixed formulation as follows. {Let $\bCt \coloneqq (\bCO , \bCT ; \bCO , \bCT ) $ be}
	$$ \bCt (\be, \bK; \tbe, \tbK )  \coloneqq   (\bCO \be, \bCT \bK; \bCO \tbe, \bCT \tbK ).$$ 
			
	\begin{pro}
		\label{var}
		Given $F \in \VLt$, find $(\Phi,\phi) \in \Vs \times \VLt$ such that
		\begin{equation*}
			\left\{  
			\begin{aligned}
				( \Phi, \Psi )_{\bCi} + ( B^{*}\Psi, \phi )_{V}  & =0, & &\forall~\Psi \in \Vs, & \\
				(B^{*}\Phi, \psi)_{V} & = - \left( F,\psi \right)_{V}, & &\forall~\psi \in \VLt,& \\
			\end{aligned}\right.
		\end{equation*}
		with
		$$
		( \Phi, \Psi )_{\bCi}  \coloneqq ( \bCOi \bN,\btau)_{S}
		+( \bCTi \bM,\bkap)_{S}
		+( \bCOi \tbN,\tbtau)_{\tS}
		+( \bCTi \tbM,\tbkap)_{\tS}.
		$$
	\end{pro}

	It is easy to verify that $( \Phi, \Psi )_{\bCi}$ is a symmetric, nonnegative bilinear form. Then it follows from Brezzi's theory (see, e.g., \cite{boffi2013mixed}) that the well-posedness of Problem $\ref{var}$ holds if and only if the following conditions are satisfied:
	\begin{enumerate}
		\item  $(\Phi, \Psi )_{\bCi}$ is bounded: There exists a constant $a>0$ such that
		$$
		|(\Phi, \Psi )_{\bCi}| \leq a \|\Phi\|_{\Vs}\|\Psi\|_{\Vs},  \quad \forall~\Phi, \Psi \in \Vs .
		$$
		\item  $(B^{*}\Psi, \psi)_{\VLt}$ is bounded: There exists a constant $b>0$ such that
		$$
		|(B^{*}\Psi, \psi)_{\VLt}| \leq b \|\Psi\|_{\Vs}\|\psi\|_\VLt,\quad  \forall~\Psi \in \Vs, \psi \in \VLt .
		$$
		\item $(\Phi, \Psi )_{\bCi}$ is coercive on the kernel of $B$: There exists a constant $\alpha>0$ such that
		$$
		( \Phi, \Psi )_{\bCi} \geq \alpha\|\Psi\|_{\Vs}^2, \quad \forall~\Psi \in \operatorname{Ker} B,
		$$
		with $$\operatorname{Ker} B=\{\Psi \in \Vs:\; (B^{*}\Psi, \psi)_{\VLt}=0 , \text{for all } \psi \in \VLt \}.$$
		\item $(B^{*}\Psi, \psi)_{\VLt}$ satisfies the inf-sup condition: There {exists} a constant $\beta>0$ such that
		$$
		\inf _{0 \neq \psi \in \VLt} \sup _{0 \neq \Psi \in \Vs } \frac{(B^{*}\Psi, \psi)_{\VLt}}{\| \Psi \|_{\Vs }\| \psi \|_\VLt } \geq \beta.
		$$
	\end{enumerate}
	These conditions are known as Brezzi's conditions with constants $a, b, \alpha$, and $\beta$. 
	
	{The following theorem shows that the solution to Problem 1 is also a solution to Problem 2. It can be utilized to prove the inf-sup condition of Problem 2.}
	
	\begin{theorem}
		\label{existence}
		Let $\phi \in \VWt$ be the solution of Problem \ref{classicalvar} for $F \in \VLt$. Then $\Phi = - \bCt B \phi \in \Vs$ and $(\Phi,\phi)$ solves Problem \ref{var}. 
	\end{theorem}
	\begin{proof}
		Recall the $L^2$ space $\VsL$ in \eqref{defVs}. It follows from $\phi \in \VWt$ that $\Phi = - \bCt B \phi \in \VsL$. Note that $\phi$ is the solution of Problem \ref{classicalvar}, then it holds
		$$
		\begin{aligned}
			( B\psi, \Phi )_{\VsL} &= - \int_{S} \nabs \bv : \bN \rd S+\int_{S} \nab^2 v_3:\bM \rd S -  \int_{\tS} \tnabs \tbv : \tbN \rd \tS+\int_{\tS} \tnab^2 \tv_3:\tbM \rd \tS \\
			& =- \left( \int_{S} \be(\bv) : \bN \rd S+\int_{S} \bK(v_3) :\bM \rd S +  \int_{\tS} \tbe(\tbv) : \tbN \rd \tS+\int_{\tS} \tbK(\tv_3) :\tbM \rd \tS \right) \\
			& = -D(\phi,\psi) = - \left( F , \psi \right)_{\VLt}, \quad  \forall~ \psi \in \VWt.
		\end{aligned}
		$$
		{Since $F \in \VLt$, this implies $\Phi\in \Vs$ by the definition of $\Vs$ in \eqref{sigma}, and that the second equation of Problem \ref{var} is satisfied by noting that $\VWt$ is dense in $\VLt$. }
		
		{{For any $\Psi \in \Vs, $ it follows from \eqref{adjointequality} that}}
		$$
		\begin{aligned}
			\left(  B^* \Psi , \phi \right)_{\VLt} &= (B \phi, \Psi )_{\VsL} \\
			& = - \int_{S} \nabs \bu : \btau \rd S+\int_{S} \nab^2 u_3:\bkap \rd S -  \int_{\tS} \tnabs \tbu : \tbtau \rd \tS+\int_{\tS} \tnab^2 \tu_3:\tbkap \rd \tS \\
			& = - \int_{S} \bCOi \bN: \btau \rd S - \int_{S} \bCTi  \bM :\bkap \rd S -   \int_{\tS}  \bCOi \tbN : \tbtau \rd \tS - \int_{\tS} \bCTi  \tbM :\tbkap \rd \tS \\
			&= - (\Phi, \Psi )_{\bCi}.     \\
		\end{aligned} 
		$$
		This proves the first equation of Problem \ref{var} and completes the proof. 
	\end{proof}

	\begin{theorem}
		\label{Brezzi}
		The Brezzi's conditions {(1)--(4)} hold for Problem \ref{var}.
	\end{theorem}
	
	\begin{proof}
		{Define $\lambda_{\min}(\bCt_i)$ and $\lambda_{\max}(\bCt_i)$ as the minimum and the maximum eigenvalue of $\bCt_i$ for $i=1,2$, respectively.} Let
		$$
		\lambda_{\min }(\bCt) \coloneqq \min \{ \lambda_{\min}(\bCO),   \lambda_{\min}(\bCT) \}, \quad \lambda_{\max }(\bCt) = \max \{ \lambda_{\max}(\bCO),   \lambda_{\max}(\bCT) \}.
		$$
		The verification of Brezzi's conditions (1)--(3) is straightforward, by assigning constants $a=1 / \lambda_{\min }(\bCt), b=1$ and $\alpha=1 / \lambda_{\max }(\bCt).$
		{{To prove the inf-sup condition (4)}}, let $\phi^\psi$ be the solution of Problem \ref{classicalvar} with the right-hand side $F^\psi=-\psi \in \VLt$ for a fixed but arbitrary $\psi \in \VLt$. {Theorem \ref{existence} shows }that $\Phi^\psi=- \bCt B \phi^\psi \in \Vs$, and $\left(   \Phi^\psi , \phi^\psi\right)$ is {a solution of Problem} \ref{var}. It follows from the second {equation} of Problem \ref{var} that
		\begin{equation}
			\label{eqq1}
				( B^* \Phi^{\psi}, \psi )_{\VLt} =-(F^\psi, \psi)_{V} = (\psi, \psi)_\VLt=\|\psi\|_V^2,
		\end{equation}
		and
		$$
		\|B^* \Phi^{\psi} \|_{\VLt} = \sup _{q \in \VLt} \frac{\left(  B^* \Phi^{\psi}, q \right)}{\|q\|_\VLt }=\sup _{q \in \VLt} \frac{(\psi, q)_\VLt}{\|q\| }=\|\psi\|_\VLt .
		$$
		{Additionally, one can obtain}
		$$
		\begin{aligned}
			\| \Phi^\psi \|_{\VsL}^2 & =\| \bCO \nab \bu^\psi \|_{0,S}^2 +  \| \bCT \nab^2 u_{3}^\psi \|_{0,S}^2 + \| \bCO \tnab \tbu^\psi \|_{0,\tS}^2 + \| \bCT \tnab^2 \tu_3^\psi \|_{0,\tS}^2 \\
			& \leq \lambda_{\max }(\bCt)D(\phi^\psi, \phi^\psi) =\lambda_{\max}(\bCt)( F^\psi, \phi^\psi )_{\VLt} .
		\end{aligned}
		$$
		{Lemma \ref{coercive} shows $\| \phi^\psi \|_\VWt \leq c \| F^\psi \|_{\VWt^*}$.}
		Let $c^{\prime}= c \lambda_{\max }(\bCt)$. Then 
$$\| \Phi^\psi \|_{\VsL}^2
		\leq \lambda_{\max }(\bCt) \| F^\psi \|_{\VWt^*} \| \phi^\psi \|_\VWt \leq c^{\prime} \| F^\psi \|_{\VWt^*}^2 \leq c^{\prime} \| F^\psi \|_{\VLt}^2=c^{\prime}\|\psi\|_\VLt^2.$$
		Hence,
		\begin{equation}
			\label{eqq2}
				\|\Phi^\psi \|_{\Vs}^2= \|\Phi^\psi \|_{\VsL}^2+ \| B^* \Phi^{\psi} \|_{ \VLt}^2 \leq(1+c^{\prime})\|\psi\|_\VLt^2. 
		\end{equation}
		{Combining \eqref{eqq1} and \eqref{eqq2} can get}
		$$
		\sup _{0 \neq \Psi \in \Vs} \frac{ \left(  B^* \Psi, \psi\right)_{\VLt} } {\|\Psi \|_{\Vs}} \geq \frac{\left( B^* \Phi^{\psi}, \psi \right)_{\VLt}}{\left\|  \Phi^{\psi} \right\|_{\Vs}} \geq(1+c^{\prime})^{-1 / 2}\| \psi\|_\VLt,
		$$
		which completes the proof with $\beta = (1+c^{\prime})^{-1 / 2}$.
	\end{proof}
	
	\begin{theorem}
		For $F \in \VLt$, Problem \ref{classicalvar} and Problem \ref{var} are equivalent in the following sense: If $\phi$ solves Problem \ref{classicalvar}, then $\Phi =- \bCt B \phi \in \Vs $ and $(\Phi, \phi)$ solves Problem \ref{var}. Conversely, if $(\Phi, \phi)$ solves Problem \ref{var}, then $\phi \in \VWt$ and solves Problem \ref{classicalvar}.
	\end{theorem}
	
	\begin{proof}
		Both Problem \ref{classicalvar} and \ref{var} are uniquely solvable due to Lemma \ref{coercive} and Theorem \ref{Brezzi}. {Thus, it suffices to demonstrate that the solution to one of the problems is also a solution to the other, which is already proven in Theorem {\ref{existence}}.}
	\end{proof}

	\subsection{Continuity conditions}
{For finite element computations, it is necessary to specify the continuity conditions of smooth functions in $\Vs$.}
	Introduce the space
	$$
	\VsC \coloneqq \{ \Psi = (\btau,\bkap; \tbtau, \tbkap); (\btau,\bkap)\in C^0(\bar S; \mathbb{S})\times C^1(\bar S; \mathbb{S}), (\tbtau,\tbkap)\in C^0( \bar{\tS}; \mathbb{S})\times C^1(\bar{\tS};\mathbb{S})\}.
	$$
	Here $\bar S$ denotes the closure of $S$.
	This theorem {shows} that the natural conditions in Problem \ref{classicalvar} become  essential conditions in Problem \ref{var}, which are hidden within the definition of $\Vs$ in \eqref{sigma} and can be explicitly extracted for smooth enough functions $\Phi$. 

	
	\begin{theorem} 
		\label{continuty}
		{Let $\Phi \in \VsL \cap \VsC$. Then $\Phi \in \Vs$ if and only if $\Phi$ satisfies} the free boundary conditions \eqref{bd2}--\eqref{bd3}, the junction conditions \eqref{jun3}--\eqref{jun4}, and {the} corner conditions \eqref{cornercondition}.
		\end{theorem}
		
		\begin{proof}
			For $\Phi \in \VsL \cap \VsC$, define the linear functional $G$ acting on $\psi \in \VWt$ as
			\begin{equation}
				\label{defG}
				(G,\psi )_{\VLt} \coloneqq ( B \psi, \Phi )_{\VLt}.
			\end{equation}
			It follows from the definition of $B$ in \eqref{defB} that 
			\begin{equation}
				\label{defIs}
				\begin{aligned}
					\quad ( B \psi, \Phi )_{\VLt} & = I_{S} + I_{\tS} \\
					&\coloneqq \left(- \int_S \bN: \nabs \bv \rd S + \int_S \bM : \nab^{2} v_{3} \rd S \right) \\
					&\qquad + \left(- \int_{\tS} \tbN : \tnabs \tbv \rd \tS + \int_{\tS} \tbM \cdot \tnab \tv_3 \rd \tS \right).  \\
				\end{aligned}
			\end{equation}
			{An integration} by parts and $\Phi \in \VsC$ show 
			\begin{equation}
				\label{Gauss1}
				\begin{aligned}
					\int_S \bN: \nabs \bv \rd S &= -\int_{S} \Div \bN \cdot \bv \rd S + \int_{\partial S} \bN \bn \cdot \bv \rd s, \\
					\int_S \bM : \nab^2 v_{3} d S &= \int_{S} \dD \bM v_{3} \rd S - \int_{\partial S} \Div \bM \cdot \bn v_{3} \rd s \\
					& \qquad + \int_{\partial S} M_{nn} \partial_{n} v_{3} \rd s - \int_{\partial S}  \partial_{t} M_{nt} v_{3} \rd s + \sum_{x \in \mathcal{V}_{\partial S}}\llbracket M_{nt} v_{3} \rrbracket_{x}. \\
				\end{aligned}
			\end{equation}
			Substituting \eqref{Gauss1} into $I_S$ {defined} in \eqref{defIs} gives
			\begin{equation}
				\label{plateS}
				I_S = \int_{S} \Div \bN \cdot \bv \rd S + \int_{S} \dD \bM v_{3} \rd S +  \RomO + \RomT + \RomR+ \RomF,
			\end{equation}
			with
			\begin{equation}
				\label{defI1}
				\begin{aligned}
					\RomO & \coloneqq -\int_{\partial S} \bN \bn \cdot \bv \rd s = -\int_{\partial S_{1}} \bN \bn \cdot \bv \rd s-\int_{\Gamma} \bN \bn \cdot \bv \rd s, \\
					\RomT & \coloneqq \int_{\partial S} M_{nn} \partial_{n} v_3  \rd s = \int_{\partial S_{1}} M_{nn} \partial_{n} v_3 \rd s+\int_{\Gamma}M_{nn} \partial_{n} v_3 \rd s, \\
					\RomR & \coloneqq -\int_{\partial S} T v_3 \rd s  = -\int_{\partial S_{1}}  T v_3  \rd s-\int_{\Gamma}  T v_3 \rd s, \\
					\RomF & \coloneqq \sum_{x \in \mathcal{V}_{\partial S}} \llbracket M_{nt} v_{3} \rrbracket_{x} = \sum_{ {x \in \mathcal{V}_{\partial S_1 }}} \llbracket M_{nt} v_{3} \rrbracket_{x} + \sum_{ {x \in \mathcal{E}_{\Gamma}}} \llbracket M_{nt} v_{3} \rrbracket_{x}.
				\end{aligned}
			\end{equation}
			{Note that \eqref{defI1} relies on} $\partial S = \partial S_0 \cup \partial S_1 \cup \Gamma$ and the definition of $\VWt$ in \eqref{W}. Similarly, due to $\partial \tS =  \partial \tS_{1} \cup \Gamma$, it holds  
			\begin{equation}
				\label{platetS}
				I_{\tS} = \int_{\tS} \tDiv \tbN \cdot \tbv \rd \tS + \int_{\tS} \tdD \tbM \tv_{3} \rd \tS + \tRomO+ \tRomT + \tRomR + \tRomF,
			\end{equation}
			with
			\begin{equation*}
				\begin{aligned}
					\tRomO & \coloneqq -\int_{\partial \tS_{1}} \tbN \tbn \cdot \tbv \rd \ts-\int_{\Gamma} \tbN \tbn \cdot \tbv \rd \ts,& \tRomT & \coloneqq \int_{\partial \tS_{1}} \tM_{\tn\tn} \tpar_{\tn} \tv_3 \rd \ts+\int_{\Gamma} \tM_{\tn \tn} \tpar_{\tn} \tv_3 \rd \ts,& \\
					\tRomR & \coloneqq -\int_{\partial \tS_{1}}  \tT \tv_3  \rd \ts-\int_{\Gamma}  \tT \tv_3 \rd \ts,&  \tRomF & \coloneqq \sum_{ {x \in \mathcal{V}_{\partial \tS_1 }}} \llbracket \tM_{\tn \ttt} \tv_{3} \rrbracket_{x} + \sum_{ {x \in \mathcal{E}_{\Gamma}}} \llbracket \tM_{\tn \ttt} \tv_{3} \rrbracket_{x}.&
				\end{aligned}
			\end{equation*}
			It should be emphasized that the integral directions on $\Gamma$ for $S$ and $\tS$ are opposite, as dictated by the definition of local coordinate systems in Section \ref{hypotheses}.
			{Recall that $\psi \in \VWt$ satisfies}
			\begin{equation}
				\label{relation}
				\partial_{n} v_3+\tpar_{\tn} \tv_3=0,\quad v_{1} = \tv_{1} \cos \theta - \tv_3 \sin \theta, \quad v_2 = -\tv_2, \quad v_3 = -\tv_1 \sin \theta - \tv_3 \cos \theta.
			\end{equation}
			Substituting \eqref{plateS} and \eqref{platetS} into \eqref{defIs} and using \eqref{relation} lead to 
			\begin{equation}
				\label{key}
				\begin{aligned}
					(G, \psi )_{\VLt} &= I_d + I_b + I_j + I_c,
				\end{aligned}
			\end{equation}
with
			\begin{equation*}
				\begin{aligned}
					I_{d} &\coloneqq \int_{S} \Div \bN \cdot \bv \rd S + \int_{S} \dD \bM v_{3} \rd S + \int_{\tS} \tDiv \tbN \cdot \tbv \rd \tS + \int_{\tS} \tdD \tbM \tv_{3} \rd \tS, \\
					I_{b} & \coloneqq - \int_{\partial S_1} \bN \bn \cdot \bv \rd s+\int_{\partial S_{1}} M_{nn} \partial_{n} v_3 \rd s - \int_{\partial S_1}  T v_{3} \rd s + \sum_{ {x \in \mathcal{V}_{\partial S_1 }}} \llbracket M_{nt} v_{3} \rrbracket_{x} \\
					&\quad \quad - \int_{\partial \tS_1} \tbN \tbn \cdot \tbv \rd \ts + \int_{\partial \tS_{1}} \tM_{\tn\tn} \tpar_{\tn} \tv_3 \rd \ts- \int_{\partial \tS_1}  \tT \tv_{3} \rd \ts + \sum_{ {x \in \mathcal{V}_{\partial \tS_1 }}} \llbracket \tM_{\tn \ttt} \tv_{3} \rrbracket_{x},\\
					I_{j}&\coloneqq \int_{\Gamma} (-N_{nn}\cos\theta - \tN_{\tn \tn} + T \sin\theta)\tv_1 \rd s + \int_{\Gamma} (N_{nt} - \tN_{\tn \ttt}) \tv_2 \rd s\\
					&\qquad + \int_{\Gamma} (N_{nn}\sin\theta + T\cos \theta - \tT) \tv_3 \rd s + \int_{\Gamma} (M_{nn}-\tM_{\tn \tn}) \tpar_{\tn}\tv_3 \rd s,\\
					I_{c} &\coloneqq \sum_{ {x \in \mathcal{E}_{\Gamma }}} \llbracket M_{nt} \rrbracket_{x}\tv_{1}(x)\sin \theta + \sum_{ {x \in \mathcal{E}_{\Gamma }}} \left( \llbracket M_{nt} \rrbracket_{x}  -  \llbracket \tM_{\tn \ttt} \rrbracket_{x} \cos \theta \right)  \tv_{3}(x).
				\end{aligned}
			\end{equation*}
			
			The only if implications follow immediately from \eqref{key}. To be specific, assume that $\Phi$ satisfies the boundary conditions \eqref{bd2}--\eqref{bd3}, \eqref{jun3}--\eqref{jun4}, and \eqref{cornercondition}, then it follows that
			$$
			(G, \psi)_{\VLt} = I_d = (B^* \Phi, \psi)_{\VLt} \leq \|B^* \Phi \|_{\VLt} \|\psi\|_{\VLt}, \quad \forall~\psi \in \VWt.
			$$
			Hence, $G$ is a linear functional on $\VLt$ and $\Phi \in \Vs$.
			
			For the if implications, assume $\Phi \in \Vs$, then the functional $G$ given by \eqref{defG} is bounded with respect to the $\VLt$-norm, for all $\psi \in \VWt$. Introduce the space
			$$
			\bar{\VWt} = \left\{ \Psi = (\bv,v_3; \tbv, \tv_3):\; \Psi \in H,\text{and } \bv,v_3,\partial_{n}v_3=0\text{ on } \partial S, \text{and } \tbv,\tv_3,\partial_{\tn} \tv_3=0\text{ on } \partial \tS \right\}.
			$$
			It is easy to verify that $\bar{\VWt} \subset \VWt$ and $\bar{\VWt}$ is dense in $\VWt$ with respect to the $V$-norm. For $\psi \in \bar{\VWt}$, it follows from \eqref{key} that
			\begin{equation}
				\label{barW}
				(G, \psi )_{\VLt} = I_d, \quad \forall~\psi \in\bar{\VWt}.
			\end{equation}
			Then one can obtain that $\eqref{barW}$ holds for $\psi \in \VWt$. This implies together with \eqref{key} that 
			$$
			I_b + I_j + I_c = 0, \quad \forall~\psi \in \VWt.
			$$
			{This and standard arguments lead to} \eqref{bd2}--\eqref{bd3}, \eqref{jun3}--\eqref{jun4}, and \eqref{cornercondition}.
		\end{proof}

		
		
		\section{Mixed finite element methods}
		
		{This section presents a framework of conforming mixed finite element methods for solving Problem \ref{var}. The well-posedness and error estimates of the discrete mixed variational problem are demonstrated. Numerical examples are provided to verify the theoretical results.}
		
		\subsection{Mixed finite element methods}
		Let $\mathcal{T}_h$ and $\tTau$ be a shape regular {triangulation} of $S$ and $\tS$, respectively, which satisfy the compatibility conditions on {junction} $\Gamma$. {Let $h$ be} {the} maximum of the diameters of all the elements $K \in \mathcal{T}_h$ and $\tK \in \tTau$. 
		Recall the space $\VLt$ in  \eqref{VLt}, the space $\VsH$ in \eqref{VHs}, and the continuty conditions for sufficient smooth functions in Theorem \ref{continuty}. Provided two conforming finite element spaces
		$$
		\begin{aligned}
			\Vs_h \subset \{\Psi_h = \left( \btauh, \bkaph ; \tbtauh, \tbkaph \right) \in \VsH: \; & \btauh|_{K} \in P_{k_1}(K,\mathbb{S}), \bkaph|_{K} \in P_{k_2}(K,\mathbb{S}),\forall K \in \Tau, \\
			&\tbtauh|_{\tK} \in P_{k_1}(\tK,\mathbb{S}), \tbkaph|_{\tK} \in P_{k_2}(\tK,\mathbb{S}),\forall \tK \in \tTau, \\
			& \Psi {\text{ satisfies the conditions in Theorem \ref{continuty}}} \},
		\end{aligned}
		$$
		and 
		$$
		\begin{aligned}
			\VLt_h = \{\psi_h =\left( \bv_h, v_{3h} ; \tbv_h, \tv_{3h} \right) \in \VLt: \;& \bv_h |_{K} \in P_{k_3}(K;\mathbb{R}^2), v_{3h} |_{K} \in P_{k_4}(K), \forall K \in \Tau, \\
			& \tbv_h |_{\tK} \in P_{k_3}(\tK;\mathbb{R}^2), \tv_{3h} |_{\tK} \in P_{k_4}(\tK),\forall \tK \in \tTau  \}
		\end{aligned}
		$$
		for some integers $k_1, k_2, k_3, k_4$ such that $B^* \Vs_h \subseteq \VLt_h$, 
		the discrete mixed formulation of Problem \ref{var} can be written as follows. 
		\begin{pro}
			\label{prodiscrete}
			Given $ F \in V$, find $\Phi_h \coloneqq (\bN_h,\bM_h; \tbN_h,\tbM_h)\in \Vs_h$ and $\phi_h \coloneqq (\bu_h,u_{3h};\tbu_h,\tu_{3h}) \in V_h$ such that
			\begin{equation}
				\label{varfem}
				\left\{   
				\begin{aligned}
					( \Phi_h, \Psi_h )_{\bCi} + ( B^{*}\Psi_h, \phi_h )_{\VLt}   & =0,&   &\forall~\Psi_h  \in \Vs_h, &\\
					( B^{*}\Phi_h, \psi_h )_{\VLt}   & = - \left( F,\psi_h \right)_{\VLt},&   &\forall~\psi_h \in V_h.& \\
				\end{aligned}\right.
			\end{equation}		
		\end{pro}

		{Instead of analyzing specific conforming finite elements, this paper provides a framework of conforming mixed finite element methods based on the following assumptions.}
		\newenvironment{assumption}[1]{\medskip\noindent{Assumption (#1):}\rmfamily\label{#1}}{\par\medskip}
		\newcommand{\cA}[1]{(\ref{#1})}
		
		%

		\begin{assumption}{A1}
			{			For all $\psi_h \in V_h$, there exists $\Phi \in \tilde{\Vs} \subset \Vs$ equipped with the norm $\| \bullet \|_{\tilde{\Vs}}$ with extra regularity compared to $\Vs$ such that $B^* \Phi = \psi_h$.}
		\end{assumption}
		\begin{assumption}{A2} There exists a Fortin operator $\Pi_h: \tilde{\Vs} \rightarrow \Vs_h$ such that  
			$
			Q_h B^{*} \Phi = B^{*}\Pi_h \Phi,
			$
			where $Q_h$ is the $L^2$ projection from $\VLt$ to $\VLt_h$, namely,  
			$$
			(B^{*}\Phi, \psi_h)_{\VLt} = (B^{*}\Pi_h \Phi, \psi_h)_{\VLt}, \quad \forall~\psi_h \in V_h.
			$$ 
			{The stability $\| \Pi_h \Phi\|_{\Vs} \leq C \|\Phi\|_{\tilde{\Vs}}$ holds with a constant $C$}.
		\end{assumption}
		\begin{theorem}
			Problem \ref{prodiscrete} is well-posedness under Assumptions $\mathrm{(A1)}$--$\mathrm{(A2)}$.
		\end{theorem}
		\begin{proof}
			It is easy to verify that $( \Phi_h, \Psi_h )_{\bCi}$ is a symmetric, nonnegative bilinear form. It remains to show the Brezzi's conditions hold {for \eqref{varfem}}, which can be derived from the continuous counterpart with the help of a Fortin operator, e.g.,\cite{boffi2013mixed}. {The proof is completed by Assumptions $\mathrm{(A1)}$--$\mathrm{(A2)}$ and Theorem \ref{Brezzi}.}
		\end{proof}
		
		Following the standard procedures in \cite{boffi2013mixed}, the well-posedness of Problem \ref{prodiscrete} allows the following error estimate.
		\begin{theorem} 
			\label{theoremConvergence}
			Let $(\Phi,\phi) \in (\Vs,\VLt)$ be the solution of Problem \ref{var} and $(\Phi_h, \phi_h) \in (\Vs_h, \VLt_{h})$ be the solution of Problem \ref{prodiscrete}. Then there exists a constant $C$ independent of mesh-size $h$ such that 
			$$
			\| \Phi -\Phi_h\|_{\Vs} + \| \phi -\phi_h\|_{\VLt} \leq C \operatorname{inf}_{\Psi_h\in \Vs_h, \psi_h \in \VLt_h} ( \| \Phi - \Psi_h \|_{\Vs} + \| \phi- \psi_{h}\|_{\VLt}).
			$$
			{Assume $\phi,\Phi$ are smooth on $S$ and $\tS$. Let $I_h \Phi$ denote the canonical interpolation of $\Phi$ into $\Vs_h$ and let $k = \min\{k_1-1,k_2-2,k_3,k_4\}$. It holds}
			$$
			\| \Phi -\Phi_h\|_{\Vs} + \| \phi -\phi_h\|_{\VLt} \leq C ( \| \Phi - I_h \Phi \|_{\Vs} + \| \phi- Q_h \phi \|_{\VLt}) \leq Ch^{k+1}.
			$$
		\end{theorem}

		{
Any finite element method that satisfies Assumptions $\mathrm{(A1)}$ and $\mathrm{(A2)}$ can be utilized to solve Problem \ref{prodiscrete}. 
Recall the continuity conditions in \eqref{jun3} and \eqref{jun4} from Theorem \ref{continuty} as follows
		$$
		 \tM_{\tn \tn}=M_{nn},~ \tN_{\tn\tn} =  -N_{nn}  \cos \theta + T \sin \theta , 
		~\tN_{\tn\ttt}= N_{nt}, ~\tT=  N_{nn}  \sin \theta + T \cos \theta,\text { on}~\Gamma.
		$$
		It can be observed that the junction conditions contain the terms $N_{nn}, N_{nt},M_{nn}$ and $T$. 
		For ease of programming implementation, two finite element pairs contain degrees of freedom of these terms are chosen, 
		for instance, the {conforming} $H(\divt,\mathbb{S})\times L^2(\mathbb{R}^2)$ mixed elements in \cite{HuZhang2014} and the {conforming} $H(\dD,\mathbb{S})\times L^2$ mixed elements in \cite{ChenHuang2022}.} 
		{Recall that} $T = (\Div \bM)\cdot \bn+\partial_t((\bM \bn)\cdot\bt)$, it follows from the construction of $\Vs_h$ that the condition $k_2=k_1+1$ is necessary for \eqref{jun4} to hold strictly.
		Figure \ref{HZ} and \ref{CH} display the degrees of freedom for these two pairs of $k_1=3, k_3=2$ and $k_2 = 4, k_4=2$, respectively. As shown on the left of Figure \ref{HZ}, except nine  $H(\divt, \mathbb{S})$-bubble functions, the other degrees of freedom of $P_3$-$H(\divt, \mathbb{S})$ include: the {values} of $\bN$ at three vertices, the normal and tangential components of $\bN\bn$ at two distinct points in the interior of each edge. They are indicated by black points and double arrows respectively. 
		On the left of Figure \ref{CH}, the degrees of freedom include the values of $\bM$ at three vertices, the values of $M_{nn}$ at three distinct points in the interior of each edge, and the values of $T$ at four distinct points in the interior of each edge (denoted by arrows). 
		\begin{figure}[ht]
			\begin{tikzpicture}[scale=2]
				\pgfmathparse{sqrt(3)}
				\let \st \pgfmathresult
				\def \m{1.3}
				\def \a{1.5}
				\draw (0,0) -- (2,0); 
				\draw (0,0) -- (1,\a); 
				\draw (2,0) -- (1,\a); 
				\def \left{0.1}
				\def  \points{(0,0),     (0.09,0), (\left/2,0.07),
					(1,\a),     (1-0.045,\a-0.07), (1+0.045,\a-0.07),
					(2,0),     (2-0.09,0),   (2-\left/2,0.07)}
				\foreach \coord in \points {\fill \coord circle (1.3 pt);}
				\node at (1,0.6) {$+9$}; 
				\node[rotate=235] at (0.3,0.6) {$\Downarrow$}; 
				\node[rotate=235] at (0.6,1.05) {$\Downarrow$}; 
				\node at (0.7,-0.08) {$\Downarrow$}; 
				\node at (1.3,-0.08) {$\Downarrow$}; 
				\node[rotate=125]  at (1.7,0.6) {$\Downarrow$}; 
				\node[rotate=125]  at (1.4,1.05) {$\Downarrow$}; 
				
				\begin{scope}[shift={(2.5cm,0)}]
					\draw (0,0) -- (2,0); 
					\draw (0,0) -- (1,\a); 
					\draw (2,0) -- (1,\a); 
					\def  \points{(0.04,0), (-0.04,0),(0.96,0),(1.04,0),(0.96,\a),(1.04,\a),(1.96,0),(2.04,0), (0.46,\a/2),(0.54,\a/2),(1.46,\a/2),(1.54,\a/2)}
					\foreach \coord in \points {\fill \coord circle (1.3 pt);}
				\end{scope}
			\end{tikzpicture}
			\caption{Left and Right are degrees of freedom for a $P_3$-$H(\divt, \mathbb{S})$ element and a discontinous vectorial $P_2$ element, respectively.}
			\label{HZ}
		\end{figure}
		\begin{figure}[ht]
			\begin{tikzpicture}[scale=2,>=latex]
				\pgfmathparse{sqrt(3)}
				\let \st \pgfmathresult
				\def \m{1.3}
				\def \a{1.5}
				\draw (0,0) -- (2,0); 
				\draw (0,0) -- (1,\a); 
				\draw (2,0) -- (1,\a); 
				\def \left{0.1}
				\def  \points{(0,0),     (0.09,0), (\left/2,0.07),
					(1,\a),     (1-0.045,\a-0.07), (1+0.045,\a-0.07),
					(2,0),     (2-0.09,0),   (2-\left/2,0.07), }
				\foreach \coord in \points {\fill \coord circle (1.3 pt);}
				\def  \pointint{(0.5,0),     (1,0),   (1.5,0),
					(0.25,\a/4),  (0.5,\a/2), (0.75,\a/4*3),
					(1.75,\a/4),  (1.5,\a/2), (1.25,\a/4*3), }
				\foreach \coord in \pointint {\fill \coord circle (1.3 pt);}
				\node at (1,0.6) {$+15$}; 
				\draw[->] (0.25,0) -- ++ (0,-0.2);
				\draw[->] (0.75,0) -- ++ (0,-0.2);
				\draw[->] (1.25,0) -- ++ (0,-0.2);
				\draw[->] (1.75,0) -- ++ (0,-0.2);
				\draw[->,rotate around={235:(0.125,0.1875)}] (0.125,0.1875) -- ++ (0,-0.2);
				\draw[->,rotate around={235:(0.375,0.5625)}] (0.375,0.5625) -- ++ (0,-0.2);
				\draw[->,rotate around={235:(0.625,0.9375)}] (0.625,0.9375) -- ++ (0,-0.2);
				\draw[->,rotate around={235:(0.875,1.3125)}] (0.875,1.3125) -- ++(0,-0.2); 
				\draw[->,rotate around={125:(1.125,1.3125)}] (1.125,1.3125) -- ++ (0,-0.2);
				\draw[->,rotate around={125:(1.375,0.9375)}] (1.375,0.9375) -- ++ (0,-0.2);
				\draw[->,rotate around={125:(1.625,0.5625)}] (1.625,0.5625) -- ++ (0,-0.2);
				\draw[->,rotate around={125:(1.875,0.1875)}] (1.875,0.1875) -- ++(0,-0.2); 
				\begin{scope}[shift={(2.5cm,0)}]
					\draw (0,0) -- (2,0); 
					\draw (0,0) -- (1,\a); 
					\draw (2,0) -- (1,\a); 
					\def  \points{(0,0), (1,0),(1,\a), (2,0), (0.5,\a/2), (1.5,\a/2)}
					\foreach \coord in \points {\fill \coord circle (1.3 pt);}
				\end{scope}
			\end{tikzpicture}
			\caption{Left and Right are degrees of freedom for a $P_4$-$H(\dD,\mathbb{S})$ element and a discontinous $P_2$ element, respectively.}
			\label{CH}
		\end{figure}
		
		For each midsurface $S$ and $\tS$, the implementation follows the similar procedures of mixed finite element methods for a single plate. As mentioned in Theorem \ref{continuty}, the only difference is to deal with the junction conditions \eqref{jun3}--\eqref{jun4} and corner conditions \eqref{cornercondition}. For the junction condition \eqref{jun3}, i.e., $M_{nn} = \tM_{\tn\tn}$ on $\Gamma$. The restriction of these functions to $\Gamma$ are polynomials of degree four in one variable. Thus, it suffices to satisfy \eqref{jun3} at the three points inside the edge together with the endpoints. For the junction conditions \eqref{jun4}, similarly, it suffices to satisfy these conditions at the points related to the degrees of freedom.
		
		\newcommand{\tx}{\underaccent{\tilde}{x}}
		\newcommand{\ty}{\underaccent{\tilde}{y}}
				\subsection{Numerical test} Two numerical experiments are carried out in this subsection to verify the theoretical results established {in Theorem \ref{theoremConvergence}. } 
		\begin{ex}
			\label{FE}
			Consider two coupled plates shown in Figure \ref{FigJunction} with two local coordinates $(\bn,\bt,\bl)$ and $(\tbn, \tbt,\tbl)$ located at the middle point of $\Gamma$. The domain in local coordinates are {$S=(-1,0)\times (-1,1)$, $\tS = (-1,0)\times (-1,1)$ }and the angle $ \theta = \pi/2$. In this example, $E = 3000, \nu = 0, e = 0.124$. 
			Let the exact solution be chosen as 
			\begin{equation}
				\label{Example1}
				\begin{aligned}
					(\bu, u_3) &=  ( -(1-x^2)^2(1-y^2)^2, 	(1-x^2)^2(1-y^2)^2,(1-x^2)^2(1-y^2)^2).\\ 
					(\tbu, \tu_3) &=  ( -(1-\tx^2)^2(1-\ty^2)^2, 	-(1-\tx^2)^2(1-\ty^2)^2,(1-\tx^2)^2(1-\ty^2)^2).\\ 
				\end{aligned}
			\end{equation}
			The external force $F$, the boundary forces, and moments in equations \eqref{bd2}--\eqref{bd3} can be determined through straightforward calculations.
			It can be verified that the solution \eqref{Example1} satisfies the clamped boundary condition \eqref{bd1} on $\partial S_0$ and junction conditions \eqref{jun1}--\eqref{jun3} on $\Gamma$. Nevertheless, this solution fails to satisfy the homogenous junction condition $N_{nt} - \tN_{\tn\ttt} =0$, while it does satisfy $N_{nt} - \tN_{\tn\ttt} = 2N_{nt} $. Consequently, the jump condition in this example is nonhomogenous. 
		\end{ex}
		
		
		In this specific example, the chosen solution \eqref{Example1} is smooth enough on each middle surface. Therefore, Theorem \ref{theoremConvergence} shows the following result 
		$$
		\| \Phi -\Phi_h\|_{\Vs} + \| \phi -\phi_h\|_{\VLt} \leq C h^{3}.
		$$
		It is observed from Table \ref{tabHZ1} and Table \ref{tabHC1}  that the convergence rates of $(\bN, \bM)$ and $(\bu, u_3)$ for the middle surface $S$ are both $O(h^3)$ in $\Sigma$ and $V$ norms, respectively. {Moreover, the rates of $(\bN, \bM)$ in $L^2$ norms are $O(h^4)$ and $O(h^5)$, both of which are optimal. }{{The mesh is uniformly refined.}} For the middle surface $\tS$, the similar conclusion can be derived from Table \ref{tabHZ2} and Table \ref{tabHC2}. These numerical results {concide with} the theoretical result in Theorem \ref{theoremConvergence}.

			\newpage
			\begin{table}[ht]
				\centering
				\begin{tabular}{ccccccc}
					\toprule[1pt]
					Mesh & $\|\bN-\bN_h\|_{0}$ & Order & $\|\bu-\bu_h\|_{0}$ & Order & $\|   \bN - \bN_h \|_{H(\divt)}$ & Order \\
					\midrule
					1     & $3.85439 \rE{+01}$ & - & $6.31904 \rE{-02}$ & - & $1.37373 \rE{+02}$ & - \\
					2     & $3.86755 \rE{+00}$ & $3.32$ & $1.02984 \rE{-02}$ & $2.62$ & $2.07569 \rE{+01}$ & $2.73$ \\
					3    & $2.69636 \rE{-01}$ & $3.84$ & $1.36836 \rE{-03}$ & $2.91$ & $2.76788 \rE{+00}$ & $2.91$ \\
					4    & $1.71509 \rE{-02}$ & $3.97$ & $1.73916 \rE{-04}$ & $2.98$ & $3.51847 \rE{-01}$ & $2.98$ \\
					5   & $1.07473 \rE{-03}$ & $4.00$ & $2.18326 \rE{-05}$ & $2.99$ & $4.41676 \rE{-02}$ & $2.99$ \\
					\bottomrule[1pt]
				\end{tabular}
				\caption{{Errors for} the in-plane pair on $S$.}
				\label{tabHZ1}
			\end{table}
			
			\begin{table}[ht]
				\centering
				\begin{tabular}{cccccccc}
					\toprule[1pt]
					Mesh & $\| \bM - \bM_h\|_{0}$ & Order & $\|u_3-u_{3h}\|_{0}$ & Order & $\| \bM- \bM_{h}\|_{H(\dD)}$ & Rate \\
					\midrule
					1     & $1.83015 \rE{-01}$ & - & $4.27046 \rE{-02}$ & - & $3.18201 \rE{+00}$ & - \\
					2     & $7.75092 \rE{-03}$ & $4.56$ & $7.24994 \rE{-03}$ & $2.56$ & $4.34361 \rE{-01}$ & $2.87$ \\
					3    & $2.70767 \rE{-04}$ & $4.84$ & $9.67331 \rE{-04}$ & $2.91$ & $5.53801 \rE{-02}$ & $2.97$ \\
					4    & $8.35573 \rE{-06}$ & $5.02$ & $1.22977 \rE{-04}$ & $2.98$ & $6.95601 \rE{-03}$ & $2.99$ \\
					5   & $2.50436 \rE{-07}$ & $5.06$ & $1.54381 \rE{-05}$ & $2.99$ & $8.70545 \rE{-04}$ & $3.00$ \\
					\bottomrule[1pt]
				\end{tabular}
				\caption{Errors for the out-of-plane pair on $S$.}
				\label{tabHC1}
			\end{table}
			
			\begin{table}[ht]
				\centering
				\begin{tabular}{cccccccc}
					\toprule[1pt]
					Mesh & $\| \tbN - \tbN_h\|_{0} $ & Order & $\|\tbu-\tbu_h\|_{0} $ & Order & $\| \tbN - \tbN_h \|_{H(\tdiv)} $ & Order \\
					\midrule
					1     & $3.04348 \rE{+01}$ & - & $6.18553 \rE{-02}$ & - & $2.23788 \rE{+02}$ & - \\
					2     & $3.16175 \rE{+00}$ & $3.27$ & $1.02746 \rE{-02}$ & $2.59$ & $4.34492 \rE{+01}$ & $2.36$ \\
					3    & $2.22673 \rE{-01}$ & $3.83$ & $1.36833 \rE{-03}$ & $2.91$ & $6.04352 \rE{+00}$ & $2.85$ \\
					4    & $1.42905 \rE{-02}$ & $3.96$ & $1.73919 \rE{-04}$ & $2.98$ & $7.75373 \rE{-01}$ & $2.96$ \\
					5   & $8.97872 \rE{-04}$ & $3.99$ & $2.18328 \rE{-05}$ & $2.99$ & $9.75505 \rE{-02}$ & $2.99$ \\
					\bottomrule[1pt]
				\end{tabular}
				\caption{Errors for the in-plane pair on {$\protect \tS$}.}
				\label{tabHZ2}
			\end{table}
			
			\begin{table}[ht]
				\centering
				\begin{tabular}{cccccccc}
					\toprule[1pt]
					Mesh & $\| \tbM - \tbM_h\|_{0} $ & Order & $\| \tu_3- \tu_{3h}\|_{0} $ & Order & $\| \tbM  -\tbM_h \|_{H(\tdD)} $ & Order \\
					\midrule
					1     & $3.88066 \rE{-01}$ & - & $4.25016 \rE{-02}$ & - & $3.18201 \rE{+00}$ & - \\
					2     & $1.61729 \rE{-02}$ & $4.58$ & $7.24375 \rE{-03}$ & $2.55$ & $4.34361 \rE{-01}$ & $2.87$ \\
					3    & $4.39241 \rE{-04}$ & $5.20$ & $9.67361 \rE{-04}$ & $2.90$ & $5.53801 \rE{-02}$ & $2.97$ \\
					4    & $1.18672 \rE{-05}$ & $5.21$ & $1.22980 \rE{-04}$ & $2.98$ & $6.95601 \rE{-03}$ & $2.99$ \\
					5   & $3.66826 \rE{-07}$ & $5.02$ & $1.54382 \rE{-05}$ & $2.99$ & $8.70545 \rE{-04}$ & $3.00$ \\
					\bottomrule[1pt]
				\end{tabular}
				\caption{Errors for the out-of-plane pair on $\protect \tS$.}
				\label{tabHC2}
			\end{table}

			\begin{ex}
				\label{ex2}
				This example considers two coupled plates with {the} rigid junction $\Gamma$ shown in Figure \ref{FigJunction2}. Given the global coordinate $(X, Y,Z)$, the left one $S$ is clamped on $\partial S_0$ while the right one is loaded along one edge of $\tpar \tS_1$ by a line density of forces $P_{z} = -1$ lb/in in global coordinate. In this example, $E = 3\times 10^7$ psi, $\nu=0$, $\alpha=30^{\circ}$, $e=0.124$ in. The length and width of the plates are shown in Figure \ref{FigJunction2}. Calculate the displacements of points $A(0,2.52\cos \alpha ), B(1/3,2.52\cos \alpha ), C(2/3,2.52\cos \alpha ),$ and $D(1,2.52\cos \alpha )$ in the global $Z$-axis. {{Figure \ref{FigMesh} shows the first mesh and the mesh is uniformly refined.}}
				
				\begin{figure}[ht]
					\centering
					\begin{tikzpicture}[>=latex]
						\draw (0,2) -- (1.5,5); 
						\draw (0,2) -- (4,1);  
						\draw (4,1) -- (5.5,4) 	
						node[pos=0, right]{$D$}
						node[pos=0.3, right]{$C$}
						node[pos=0.6, right]{$B$}
						node[pos=0.9, right]{$A$};
						\draw (5.5,4) -- (1.5,5);
						
						\filldraw [black] (4,1) circle (2pt);
						\filldraw [black] (4.5,2) circle (2pt);
						\filldraw [black] (5,3) circle (2pt);
						\filldraw [black] (5.5,4) circle (2pt);
						
						\draw (0,2) -- (1.5,5) node[pos=0.5, left]{$\Gamma$};;
						\draw (0,2) -- (-4,1);
						\draw (-4,1) -- (-2.5,4) node[pos=0.5, right]{$\partial S_0$};
						\draw (-2.5,4) -- (1.5,5);
						\foreach \x in {1,1.3,...,4.3}{
							\pgfmathsetmacro{\resultA}{\x -1}
							\pgfmathsetmacro{\step}{\resultA /0.3}
							\pgfmathsetmacro{\steps}{\step*0.15}
							\pgfmathsetmacro{\resultB}{-4 + \steps}
							\draw (\resultB,\x) -- ++ ( 0, 0.25);}
						
						\foreach \x in {1,1.5,...,4.3}{
							\pgfmathsetmacro{\resultA}{\x -1}
							\pgfmathsetmacro{\step}{\resultA /0.3}
							\pgfmathsetmacro{\steps}{\step*0.15}
							\pgfmathsetmacro{\resultB}{ 4 + \steps}
							\draw[<-]  (\resultB,\x) -- ++ (0, 1);}
						
						\draw (-2,1) coordinate (A) -- (-4,1) coordinate (B)
						-- (-2,1.5) coordinate (C)
						pic ["$\alpha$", draw, angle eccentricity=2] {angle};
						\draw (2,1.5) coordinate (A) -- (4,1) coordinate (B)
						-- (2,1) coordinate (C)
						pic ["$\alpha$", draw, angle eccentricity=2] {angle};
						
						\draw [|<->|] (0.2,2) -- node[right=0.1mm] {1.0 in} (1.7,5);
						\draw [|<->|] (1.5,5.2) -- node[above=0.1mm] {1.26 in} (5.5,4.2);
						\draw [|<->|] (-2.5,4.2) -- node[above=0.1mm] {1.26 in} (1.5,5.2);
						
						\draw [line width=0.2mm, -{Stealth[length= 3mm, open]}]
						(-2.5,4) -- (-2.5,5.5) node[pos=0.7, right]{$Z$};
						\draw [line width=0.2mm, -{Stealth[length= 3mm, open]}]
						(5.5,4) -- (7, 4) node[pos=0.7, below]{$Y$};		
						\draw [line width=0.2mm, -{Stealth[length= 3mm, open]}]
						(-4,1) -- (-4.5, 0) node[pos=0.7, right]{$X$};
					\end{tikzpicture}
					\caption{The coupled  plates with {the} rigid junction of Example \ref{ex2}.}
					\label{FigJunction2}
				\end{figure}
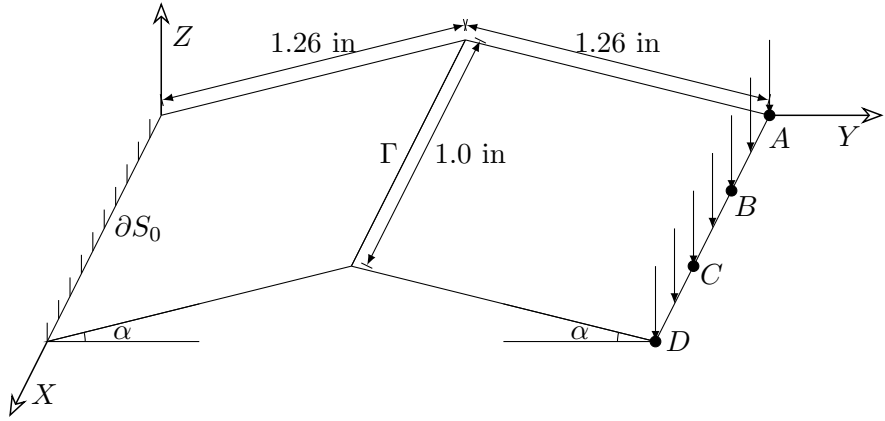

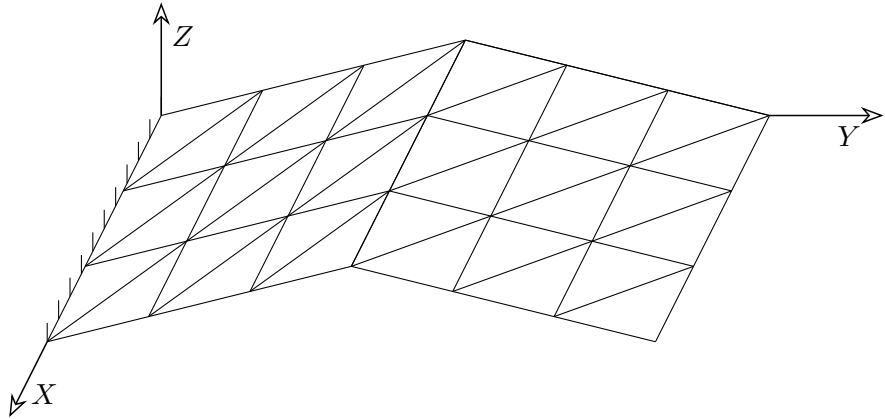
\begin{figure}[ht]
					\centering
					\begin{tikzpicture}[>=latex]
						\draw (0,2) -- (1.5,5); 
						\draw (0,2) -- (4,1);  
						\draw (4,1) -- (5.5,4);
						\draw (5.5,4) -- (1.5,5);
						\draw (5.5,4) -- (1.5,5);
						
						\draw (4/3,5/3) -- (1.5+4/3, 14/3);
						\draw (8/3,4/3) -- (1.5+8/3, 13/3);
						\draw (0.5,3) -- (4.5,2);
						\draw (1,4)     -- (5,3);
						\draw (1,4)     -- (1.5+4/3, 14/3);
						\draw (0.5,3) -- (1.5+8/3, 13/3);
						\draw (0,2)    -- (5.5,4);
						\draw (4/3,5/3) -- (5,3);
						\draw (8/3,4/3) -- (4.5, 2);
						
						\draw(-8/3,4/3) -- (-2.5+4/3,13/3);
						\draw (-4/3,5/3) -- (-2.5+8/3,14/3);
						\draw (-3.5,2) -- (0.5,3);
						\draw (-3,3) -- (1,4);
						\draw (-3,3) -- (-2.5+4/3, 13/3);
						\draw (-3.5,2) -- (-2.5+8/3,14/3);
						\draw (-4,1) -- (1.5,5);
						\draw (-8/3,4/3) -- (1,4);
						\draw (-4/3,5/3) -- (0.5,3);
						
						\draw (0,2) -- (1.5,5) ;
						\draw (0,2) -- (-4,1);
						\draw (-4,1) -- (-2.5,4) ;
						\draw (-2.5,4) -- (1.5,5);
						\foreach \x in {1,1.3,...,4.3}{
							\pgfmathsetmacro{\resultA}{\x -1}
							\pgfmathsetmacro{\step}{\resultA /0.3}
							\pgfmathsetmacro{\steps}{\step*0.15}
							\pgfmathsetmacro{\resultB}{-4 + \steps}
							\draw (\resultB,\x) -- ++ ( 0, 0.25);}

						

						\draw [line width=0.2mm, -{Stealth[length= 3mm, open]}]
						(-2.5,4) -- (-2.5,5.5) node[pos=0.7, right]{$Z$};
						\draw [line width=0.2mm, -{Stealth[length= 3mm, open]}]
						(5.5,4) -- (7, 4) node[pos=0.7, below]{$Y$};		
						\draw [line width=0.2mm, -{Stealth[length= 3mm, open]}]
						(-4,1) -- (-4.5, 0) node[pos=0.7, right]{$X$};
					\end{tikzpicture}
					\caption{The first mesh of Example \ref{ex2}.}
					\label{FigMesh}
				\end{figure}

				\begin{table}[ht]
					\centering
					\begin{tabular}{ccccc}
						\toprule[1pt]
						& \multicolumn{4}{c}{Displacement in the global $Z$-axis }\\
						\hline
						Mesh   &  A &  B  & C  &  D\\
						\hline
						1   & -8.35998E-4 & -8.38402E-4 & -8.40832E-4 &  -8.43207E-4 \\
						2   & -8.37679E-4 &-8.38905E-4 & -8.40137E-4&  -8.41361E-4\\
						3  & -8.38579E-4 & -8.39198E-4 & -8.39819E-4&  -8.40438E-4\\
						4 & -8.39040E-4 & -8.39351E-4 & -8.39663E-4&  -8.39974E-4\\
						5 & -8.39273E-4 & -8.39429E-4 & -8.39585E-4&  -8.39741E-4\\
						\toprule[1pt]
					\end{tabular}
					\caption{{Results by mixed finite element pairs.}}
					\label{tabu3mix}
				\end{table}

				\begin{table}[ht]
					\centering
					\begin{tabular}{ccccc}
						\toprule[1pt]
						& \multicolumn{4}{c}{Displacement in the global $Z$-axis }\\
						\hline
						Mesh   &  A &  B  & C  &  D\\
						\hline
						1   & -8.56525E-4 & -8.56744E-4 & -8.56745E-4 &  -8.56525E-4 \\
						2   &-8.43746E-4 & -8.43816E-4 & -8.43816E-4&  -8.43746E-4\\
						3  & -8.40565E-4 & -8.40584E-4 &-8.40584E-4&  -8.40564E-4\\
						4 & -8.39771E-4 & -8.39776E-4   & -8.39776E-4&  -8.39771E-4\\
						5 & -8.39573E-4 & -8.39574E-4   & -8.39574E-4&  -8.39573E-4\\
						\toprule[1pt]
					\end{tabular}
					\caption{{Results by noncomforming finite element pairs.}}
					\label{tabu3non}
				\end{table}
				

			\end{ex}	
			In \cite{bernadou1989numerical}, the conforming finite element pair, reduced Hermite element and reduced HCT element, are used. It provides the result $Z_{B} = -8.37166$E-4 in and $Z_{C} = -8.37274$E-4 in on the first mesh. Table \ref{tabu3mix} shows the results of the mixed pairs at the beginning of this subsection. It follows from Table \ref{tabu3mix} that the displacements are convergence to $-8.395$E-4 in. Table \ref{tabu3non} shows the results of the nonconforming finite element pair, linear Lagrange element and Morley element, which can be found in \cite{lie1992mathematical}. It can be seen from Table \ref{tabu3non} that the displacements are convergence to $-8.396$E-4 in. {These results show the displacements of the cantilever plate by mixed method are in good agreement with the results from the methods based on displacements.	}
			
			\bibliographystyle{plain}
			\bibliography{reference}

\begin{thebibliography}{10}

\bibitem{adams2004mixed}
Scot Adams and Bernardo Cockburn.
\newblock A mixed finite element method for elasticity in three dimensions.
\newblock {\em J. Sci. Comput.}, 25:515--521, 2005.

\bibitem{arnold2005rectangular}
Douglas Arnold and Gerard Awanou.
\newblock Rectangular mixed finite elements for elasticity.
\newblock {\em Math. Models Methods Appl. Sci.}, 15(09):1417--1429, 2005.

\bibitem{arnold2008finite}
Douglas Arnold, Gerard Awanou, and Ragnar Winther.
\newblock Finite elements for symmetric tensors in three dimensions.
\newblock {\em Math. Comp.}, 77(263):1229--1251, 2008.

\bibitem{ArnoldFoundations}
Douglas Arnold and Kaibo Hu.
\newblock Complexes from complexes.
\newblock {\em Found. Comput. Math.}, 21(6):1739--1774, 2021.

\bibitem{ArnoldWinther2002}
Douglas Arnold and Ragnar Winther.
\newblock Mixed finite elements for elasticity.
\newblock {\em Numer. Math.}, 92(3):401--419, 2002.

\bibitem{awanou2012two}
Gerard Awanou.
\newblock Two remarks on rectangular mixed finite elements for elasticity.
\newblock {\em J. Sci. Comput.}, 50:91--102, 2012.

\bibitem{Structuralofshells1982}
Klaus~Jurgen Bathe.
\newblock {\em Finite element procedures in engineering analysis}.
\newblock Prentice Hall, Inc., Englewood Cliffs, New Jersey, 1982.

\bibitem{bernadou1989numerical}
Michel Bernadou, S{\'e}verine Fayolle, and Fran{\c{c}}oise L{\'e}n{\'e}.
\newblock Numerical analysis of junctions between plates.
\newblock {\em Comput. Methods Appl. Mech. Engrg.}, 74(3):307--326, 1989.

\bibitem{boffi2013mixed}
Daniele Boffi, Franco Brezzi, and Michel Fortin.
\newblock {\em Mixed finite element methods and applications}, volume~44.
\newblock Springer, 2013.

\bibitem{ChenHuang2022}
Long Chen and Xuehai Huang.
\newblock Finite elements for {${\rm div\,div}$} conforming symmetric tensors
  in three dimensions.
\newblock {\em Math. Comp.}, 91(335):1107--1142, 2022.

\bibitem{chen2023new}
Long Chen and Xuehai Huang.
\newblock A new div-div-conforming symmetric tensor finite element space with
  applications to the biharmonic equation.
\newblock {\em Math. Comp.}, to appear, 2024.

\bibitem{chen2011conforming}
Shaochun Chen and Yana Wang.
\newblock Conforming rectangular mixed finite elements for elasticity.
\newblock {\em J. Sci. Comput.}, 47:93--108, 2011.

\bibitem{ciarlet1990new}
Philippe~G. Ciarlet.
\newblock A new class of variational problems arising in the modeling of
  elastic multi-structures.
\newblock {\em Numer. Math.}, 57(1):547--560, 1990.

\bibitem{Ciarlet1997}
Philippe~G. Ciarlet.
\newblock {\em Mathematical elasticity. {V}olume {II}. {T}heory of plates},
  volume~85.
\newblock Society for Industrial and Applied Mathematics (SIAM), Philadelphia,
  PA, 1997.

\bibitem{Ciarlet1989JMPA}
Philippe~G. Ciarlet, Herv\'{e} Le~Dret, and Robert Nzengwa.
\newblock Junctions between three-dimensional and two-dimensional linearly
  elastic structures.
\newblock {\em J. Math. Pures Appl.}, 68(3):261--295, 1989.

\bibitem{Fengshi1996}
Kang Feng and Zhongci Shi.
\newblock {\em Mathematical theory of elastic structures}.
\newblock Springer-Verlag, Berlin; Science Press Beijing, Beijing, 1996.
\newblock Translated from the 1981 Chinese original, Revised by the authors.

\bibitem{fuhrer2023mixed}
Thomas F{\"u}hrer and Norbert Heuer.
\newblock Mixed finite elements for {K}irchhoff-{L}ove plate bending.
\newblock {\em arXiv preprint arXiv:2305.08693}, 2023.

\bibitem{fuhrer2019ultraweak}
Thomas F{\"u}hrer, Norbert Heuer, and Antti Niemi.
\newblock An ultraweak formulation of the {K}irchhoff-{L}ove plate bending
  model and {DPG} approximation.
\newblock {\em Math. Comp.}, 88(318):1587--1619, 2019.

\bibitem{huang2011NMPDE}
Ling Guo and Jianguo Huang.
\newblock Adini-{$Q_1$}-{$P_3$} {FEM} for general elastic multi-structure
  problems.
\newblock {\em Numer. Methods Partial Differential Equations},
  27(5):1092--1112, 2011.

\bibitem{Hu2015}
Jun Hu.
\newblock Finite element approximations of symmetric tensors on simplicial
  grids in {$\Bbb R^n$}: the higher order case.
\newblock {\em J. Comput. Math.}, 33(3):283--296, 2015.

\bibitem{hurect2015}
Jun Hu.
\newblock A new family of efficient conforming mixed finite elements on both
  rectangular and cuboid meshes for linear elasticity in the symmetric
  formulation.
\newblock {\em SIAM J. Numer. Anal.}, 53(3):1438--1463, 2015.

\bibitem{hu2022new}
Jun Hu, Yizhou Liang, Rui Ma, and Min Zhang.
\newblock New conforming finite element divdiv complexes in three dimensions.
\newblock {\em arXiv:2204.07895}, 2022.

\bibitem{HuMaZhang2021}
Jun Hu, Rui Ma, and Min Zhang.
\newblock A family of mixed finite elements for the biharmonic equations on
  triangular and tetrahedral grids.
\newblock {\em Sci. China Math.}, 64(12):2793--2816, 2021.

\bibitem{hurectany2014}
Jun Hu, Hongying Man, and Shangyou Zhang.
\newblock A simple conforming mixed finite element for linear elasticity on
  rectangular grids in any space dimension.
\newblock {\em J. Sci. Comput.}, 58(2):367--379, 2014.

\bibitem{HuZhang2014}
Jun Hu and Shangyou Zhang.
\newblock A family of conforming mixed finite elements for linear elasticity on
  triangular grids.
\newblock {\em arXiv:1406.7457}, 2015.

\bibitem{HuZhang2015}
Jun Hu and Shangyou Zhang.
\newblock A family of symmetric mixed finite elements for linear elasticity on
  tetrahedral grids.
\newblock {\em Sci. China Math.}, 58(2):297--307, 2015.

\bibitem{HuZhang2016}
Jun Hu and Shangyou Zhang.
\newblock Finite element approximations of symmetric tensors on simplicial
  grids in {$\Bbb{R}^n$}: the lower order case.
\newblock {\em Math. Models Methods Appl. Sci.}, 26(9):1649--1669, 2016.

\bibitem{HuangGuoShi_CMA_2007}
Jianguo Huang, Ling Guo, and Zhongci Shi.
\newblock A finite element method for investigating general elastic
  multi-structures.
\newblock {\em Comput. Math. Appl.}, 53(12):1867--1895, 2007.

\bibitem{jianguo2005some}
Jianguo Huang, Zhongci Shi, and Yifeng Xu.
\newblock Some studies on mathematical models for general elastic
  multi-structures.
\newblock {\em Sci. China Ser. A}, 48:986--1007, 2005.

\bibitem{huang2006finite}
Jianguo Huang, Zhongci Shi, and Yifeng Xu.
\newblock Finite element analysis for general elastic multi-structures.
\newblock {\em Sci. China Ser. A}, 49(1):109, 2006.

\bibitem{KolpaAndriaMark_ZZAMM_2015}
Alexander~G. Kolpakov, Igor~V. Andrianov, and Bernd Markert.
\newblock Asymptotic decomposition in the problem of joined elastic plates.
\newblock {\em ZAMM Z. Angew. Math. Mech.}, 95(11):1268--1281, 2015.

\bibitem{kozlov1999asymptotic}
Vladimir Kozlov, VG~Maz${'}$\t{ia}, and Alexander~B Movchan.
\newblock {\em Asymptotic analysis of fields in multi-structures}.
\newblock Oxford mathematical monographs, 1999.

\bibitem{LaiHuang_DCDSSB_2009}
Junjiang Lai and Jianguo Huang.
\newblock A finite element method for vibration analysis of elastic plate-plate
  structures.
\newblock {\em Discrete Contin. Dyn. Syst. Ser. B}, 11(2):387--419, 2009.

\bibitem{LaiHuangShi2010}
Junjiang Lai, Jianguo Huang, and Zhongci Shi.
\newblock A lumped mass finite element method for vibration analysis of elastic
  plate-plate structures.
\newblock {\em Sci. China Math.}, 53(6):1453--1474, 2010.

\bibitem{Paulydividv}
Dirk Pauly and Walter Zulehner.
\newblock The div{D}iv-complex and applications to biharmonic equations.
\newblock {\em Appl. Anal.}, 99(9):1579--1630, 2020.

\bibitem{Walter2018}
Katharina Rafetseder and Walter Zulehner.
\newblock A decomposition result for {K}irchhoff plate bending problems and a
  new discretization approach.
\newblock {\em SIAM J. Numer. Anal.}, 56(3):1961--1986, 2018.

\bibitem{RafetZulehCMAME2019}
Katharina Rafetseder and Walter Zulehner.
\newblock A new mixed approach to {K}irchhoff-{L}ove shells.
\newblock {\em Comput. Methods Appl. Mech. Engrg.}, 346:440--455, 2019.

\bibitem{timoshenko1959theory}
Stephen~P Timoshenko and S~Woinowsky-Krieger.
\newblock {\em Theory of plates and shells}, volume~2.
\newblock McGraw-hill New York, 1959.

\bibitem{lie1992mathematical}
Lieheng Wang.
\newblock A mathematical model of coupled plates and its finite element method.
\newblock {\em Comput. Methods Appl. Mech. Engrg.}, 99(1):43--59, 1992.

\bibitem{lie1993mathematical}
Lieheng Wang.
\newblock Mathematical model of coupled plates meeting at an angle $0< \theta<
  \pi$ and its finite element method.
\newblock {\em Numer. Methods Partial Differential Equations}, 9(1):13--22,
  1993.

\bibitem{YeZhangRAM2022}
Xiu Ye and Shangyou Zhang.
\newblock A family of {H}-div-div mixed triangular finite elements for the
  biharmonic equation.
\newblock {\em Results Appl. Math.}, 15:100318, 2022.

\end{thebibliography}

		\end{document}